\documentclass[a4paper]{amsart}
\usepackage{graphicx}
\usepackage{amsmath}
\usepackage{amssymb}
\usepackage{oldgerm}
\usepackage[all]{xy}

\newcommand{\Hom}{\operatorname{Hom}\nolimits}
\newcommand{\stHom}{\operatorname{\underline{\Hom}}\nolimits}

\renewcommand{\mod}{\operatorname{mod}\nolimits}
\newcommand{\stmod}{\operatorname{\underline{\mod}}\nolimits}

\newcommand{\Coker}{\operatorname{Coker}\nolimits}

\newcommand{\Ann}{\operatorname{Ann}\nolimits}

\newcommand{\Ext}{\operatorname{Ext}\nolimits}

\newcommand{\Maxspec}{\operatorname{MaxSpec}\nolimits}

\newcommand{\HH}{\operatorname{HH}\nolimits}

\newcommand{\rad}{\operatorname{rad}\nolimits}
\newcommand{\m}{\operatorname{\mathfrak{m}}\nolimits}
\newcommand{\az}{\operatorname{\mathfrak{a}}\nolimits}
\newcommand{\ra}{\operatorname{\mathfrak{r}}\nolimits}
\newcommand{\La}{\operatorname{\Lambda}\nolimits}
\newcommand{\Z}{\operatorname{Z}\nolimits}
\newcommand{\cx}{\operatorname{cx}\nolimits}

\newcommand{\op}{\operatorname{op}\nolimits}
\newcommand{\V}{\operatorname{V}\nolimits}
\newcommand{\A}{\operatorname{A}\nolimits}

\newcommand{\Tr}{\operatorname{Tr}\nolimits}

\newcommand{\e}{\operatorname{e}\nolimits}

\newcommand{\Lae}{\operatorname{\Lambda^{\e}}\nolimits}
\newcommand{\lcm}{\operatorname{lcm}\nolimits}

\newcommand{\ev}{\operatorname{ev}\nolimits}

\newtheorem{theorem}{Theorem}[section]

\newtheorem{corollary}[theorem]{Corollary}

\newtheorem{lemma}[theorem]{Lemma}
\newtheorem{proposition}[theorem]{Proposition}
\theoremstyle{definition}
\newtheorem{definition}[theorem]{Definition}
\theoremstyle{definition}
\newtheorem*{example}{Example}

\theoremstyle{definition}

\theoremstyle{remark}
\newtheorem*{remark}{Remark}
\theoremstyle{definition}
\newtheorem*{assumption}{Assumption}

\begin{document}
\title{Twisted support varieties}
\author{Petter Andreas Bergh}
\address{Petter Andreas Bergh \newline Institutt for matematiske fag \\
NTNU \\ N-7491 Trondheim \\ Norway}

\email{bergh@math.ntnu.no}



\maketitle

\begin{abstract}
We define and study twisted support varieties for modules over an
Artin algebra, where the twist is induced by an automorphism of the
algebra. Under a certain finite generation hypothesis we show that
the twisted variety of a module satisfies Dade's Lemma and is one
dimensional precisely when the module is periodic with respect to
the twisting automorphism. As a special case we obtain results on $D
\Tr$-periodic modules over Frobenius algebras.
\end{abstract}

\section{Introduction}\label{intro}

In this paper we define and study a new type of cohomological
support varieties, called \emph{twisted support varieties}, for
modules over Artin algebras, varieties sharing many of the
properties of those defined for Artin algebras in \cite{Erdmann1}
and \cite{Snashall}. In those papers the underlying geometric object
used to define support varieties was the Hochschild cohomology ring
of an algebra, whereas our geometric object is a ``twisted"
Hochschild cohomology ring, the twist being induced by an
automorphism of the algebra.

By introducing a finite generation hypothesis similar to the one
used in \cite{Bergh1} we are able to relate the dimension of the
twisted variety of a module to the polynomial growth of the lengths of
the modules in its minimal projective resolution. In particular
Dade's Lemma holds, that is, the twisted variety of a module is
trivial if and only if the module has finite projective dimension.
This property is highly desirable in any cohomological variety
theory. Moreover, the modules whose twisted varieties are one
dimensional are precisely those which are periodic with respect to
the twisting automorphism, a concept defined in Section \ref{var}.
As a special case we obtain results on $D \Tr$-periodic modules over
Frobenius algebras.

Both when studying the twisted Hochschild cohomology ring and
twisted support varieties we illustrate the theory with examples
taken from the well known representation theory of the four
dimensional Frobenius algebra
$$k \langle x,y \rangle / ( x^2, xy+qyx, y^2 ),$$
where $k$ is a field and $q \in k$ is not a root of unity. Indeed,
detecting the $D \Tr$-periodic modules over this algebra by means of
support varieties was the motivation for this paper in the first
place.

\section{Preliminaries}\label{pre}

Throughout we let $k$ be a commutative Artin ring and $\La$ an
indecomposable Artin $k$-projective $k$-algebra with Jacobson
radical $\ra$. Unless otherwise specified all modules considered are
finitely generated left modules. We denote by $\mod \La$ the
category of all finitely generated left $\La$-modules, and we fix a
nonzero module $M \in \mod \La$ with minimal projective resolution
$$\mathbb{P} \colon \cdots \to P_2 \xrightarrow{d_2} P_1
\xrightarrow{d_1} P_0 \xrightarrow{d_0} M \to 0.$$

Let $F \colon \mod \La \to \mod \La$ be an exact $k$-functor. It follows
from a theorem of Watts (see \cite[Corollary 3.34]{Rotman}) that
there exists a bimodule $Q$ having the property that $F$ is
naturally equivalent to the functor $Q \otimes_{\La}-$, and that $Q$
may be chosen to be $F(\La)$. Since $Q$ is in $\mod \La$ it has
finite $k$-length and is therefore finitely generated also as a
right module, and this implies that $Q$ is projective as a right
module (it is flat). Suppose $Q$ is projective also as a left module
(but not necessarily as a bimodule). Then if $P$ is a projective
$\La$-module the functor $\Hom_{\La}(F(P),-)$, being naturally
isomorphic to $\Hom_{\La}(P,\Hom_{\La}(Q,-))$, is exact, hence
$F(P)$ is projective.

\sloppy For such a functor $F$ and a positive integer $t \in
\mathbb{N}$, define a homogeneous product in the graded $k$-module
$$\Ext_{\La}^{t*}(F^*(M),M) = \bigoplus_{n=0}^{\infty}
\Ext_{\La}^{tn} (F^n(M),M)$$ as follows; if $\eta$ and $\theta$ are
two homogeneous elements in $\Ext_{\La}^{t*} (F^*(M),M)$, then
$$\eta \theta \stackrel{\text{def}}{=} \eta \circ F^{|\eta|/t}
(\theta),$$ where ``$\circ$" denotes the Yoneda product. In other
words, if $\eta$ and $\theta$ are given as exact sequences
\begin{eqnarray*}
\eta \colon 0 \to M \to A_{tn-1} \to \cdots \to
A_0 \to F^n(M) \to 0 \\
\theta \colon 0 \to M \to B_{tm-1} \to \cdots \to B_0 \to F^m(M)\to
0,
\end{eqnarray*}
then $\eta \theta$ is given as the exact sequence obtained from the
Yoneda product of the sequence $\eta$ with the sequence
$$F^n(\theta) \colon 0 \to F^n(M) \to
F^n(B_{tm-1}) \to \cdots \to F^n(B_0) \to F^{m+n}(M) \to 0.$$
Furthermore, for a $\La$-module $N$, define a homogeneous right
scalar action on $\Ext_{\La}^{t*}(F^*(M),N)$ from
$\Ext_{\La}^{t*}(F^*(M),M)$ as follows; if $\mu \in
\Ext_{\La}^{t*}(F^*(M),N)$ and $\eta \in \Ext_{\La}^{t*}(F^*(M),M)$
are homogeneous elements, then $\mu \eta \stackrel{\text{def}}{=}
\mu \circ F^{|\mu|/t}(\eta)$.

\begin{lemma}\label{ring}
Extending the homogeneous product and scalar product defined above
to all graded elements makes $\Ext_{\La}^{t*}(F^*(M),M)$ into a
graded $k$-algebra and $\Ext_{\La}^{t*}(F^*(M),N)$ into a graded
right $\Ext_{\La}^{t*}(F^*(M),M)$-module.
\end{lemma}

\begin{proof}
The product in $\Ext_{\La}^{t*}(F^*(M),M)$ is associative, and the
right distributive law holds; if $\theta_1,\theta_2$ and $\eta$ are
homogeneous elements such that $\theta_1$ and $\theta_2$ are of the
same degree $tn$, then
$$(\theta_1+\theta_2)\eta =
(\theta_1+\theta_2) \circ F^n(\eta) = \theta_1 \circ F^n(\eta) +
\theta_2 \circ F^n(\eta) = \theta_1 \eta + \theta_2 \eta.$$ Now
suppose the degree of $\eta$ is $tm$. We can represent $\theta_i$ as
a map
$$\Omega_{\La}^{tn}(F^n(M)) \xrightarrow{f_{\theta_i}} M,$$
and since the functor $F^m$ preserves direct limits (in particular
pushouts) and projective resolutions (though it may not preserve
\emph{minimal} resolutions), we see that the map
$$F^m( \Omega_{\La}^{tn}(F^n(M))) \xrightarrow{F^m(f_{\theta_i})}
F^m(M)$$ represents $F^m(\theta_i)$. As $F^m$ is additive we have
$F^m(f_{\theta_1}+f_{\theta_2})=F^m(f_{\theta_1})+F^m(f_{\theta_2})$
as elements of $\Hom_{\La}(F^m( \Omega_{\La}^{tn}(F^n(M))),F^m(M))$,
and this implies that
$F^m(\theta_1+\theta_2)=F^m(\theta_1)+F^m(\theta_2)$. Therefore the
left distributive law also holds for the homogeneous product;
\begin{eqnarray*}
\eta (\theta_1+\theta_2) &=& \eta \circ \left (
F^m(\theta_1+\theta_2) \right ) = \eta \circ \left (
F^m(\theta_1)+F^m(\theta_2) \right ) \\
&=& \eta \circ F^m(\theta_1) + \eta \circ F^m(\theta_2) = \eta
\theta_1 + \eta \theta_2.
\end{eqnarray*}
Extending this homogeneous product in the natural way we see that
$\Ext_{\La}^{t*}(F^*(M),M)$ becomes a graded $k$-algebra, and
similar arguments show that $\Ext_{\La}^{t*}(F^*(M),N)$ is a graded
right $\Ext_{\La}^{t*}(F^*(M),M)$-module.
\end{proof}

\begin{example}
Suppose $k$ is a field and $\La$ is selfinjective, and consider a
bimodule $B$ in the stable category $\stmod \Lae$ of finitely
generated $\Lae$-modules modulo projectives. The Auslander-Reiten
translation $\tau_{\Lae} = D \Tr$ is a self-equivalence on $\stmod
\Lae$, and we can endow the graded $k$-vector space
$$\mathbb{A}(B, \tau_{\Lae}) = \Hom_{\Lae} (B,B) \oplus
\bigoplus_{i=1}^{\infty} \stHom_{\Lae} ( \tau_{\Lae}^i(B),B)$$ with
a product as follows; for two homogeneous elements $f \in
\mathbb{A}(B, \tau_{\Lae})_m$ and $g \in \mathbb{A}(B,
\tau_{\Lae})_n$ we define $fg$ to be the composition $f \circ
\tau_{\Lae}^m(g) \in \mathbb{A}(B, \tau_{\Lae})_{m+n}$. In this way
$\mathbb{A}(B, \tau_{\Lae})$ becomes a graded $k$-algebra. Now for
any selfinjective Artin algebra $\Gamma$ the functors
$\tau_{\Gamma}^i$ and $\Omega_{\Gamma}^{2i} \mathcal{N}^i$ are
isomorphic by \cite[Proposition IV.3.7]{Auslander}, where
$\mathcal{N}$ is the Nakayama functor $D \Hom_{\Gamma}(-, \Gamma)$.
Therefore the orbit algebra $\mathbb{A}(B, \tau_{\Lae})$ of the
bimodule $B$ has the form
$$\Hom_{\Lae} (B,B) \oplus \bigoplus_{i=1}^{\infty} \stHom_{\Lae} (
\Omega_{\Lae}^{2i}( \mathcal{N}^iB ),B),$$ giving an isomorphism
$$\mathbb{A}(B, \tau_{\Lae}) \simeq \bigoplus_{i=0}^{\infty}
\Ext_{\Lae}^{2i} ( \mathcal{N}^iB,B)$$ of graded $k$-algebras.

The algebra $\mathbb{A}(\La, \tau_{\Lae})$ is called the
\emph{Auslander-Reiten orbit algebra} of $\La$, and these algebras
have been extensively studied by Z.\ Pogorza{\l}y. In
\cite{Pogorzaly1} it was shown that they are invariant under stable
equivalences of Morita type between symmetric algebras; if $\Gamma$
and $\Delta$ are finite dimensional symmetric $K$-algebras (where
$K$ is a field) stably equivalent of Morita type, then
$\mathbb{A}(\Gamma, \tau_{\Gamma^{\e}})$ and $\mathbb{A}(\Delta,
\tau_{\Delta^{\e}})$ are isomorphic $K$-algebras. This was
generalized in \cite{Pogorzaly2} to arbitrary finite dimensional
selfinjective algebras. In \cite{Pogorzaly4} the Auslander-Reiten
orbit algebras of a class of finite dimensional basic connected
selfinjective Nakayama $K$-algebras were computed. Namely, it was
shown that if $\Gamma$ is such an algebra of
$\tau_{\Gamma^{\e}}$-period $1$, then $\mathbb{A}(\Gamma,
\tau_{\Gamma^{\e}}) \simeq K[x]$ if $\Gamma$ is a radical square
zero algebra, and if not then there exists a natural number $t$ such
that $\mathbb{A}(\Gamma, \tau_{\Gamma^{\e}}) \simeq K[x,y]/(y^t)$.
In \cite{Pogorzaly3} $\tau$-periodicity was investigated using
similar techniques as was used in \cite{Schulz} to study
syzygy-periodicity.
\end{example}

\section{Twisted Hochschild cohomology}\label{hoc}

The underlying geometric object used in \cite{Erdmann1} and
\cite{Snashall} to define support varieties was the Hochschild
cohomology ring $\HH^* ( \La, \La ) = \Ext_{\Lae}^*(\La,\La) =
\bigoplus_{n=0}^{\infty} \Ext_{\Lae}^n(\La,\La)$ of $\La$ (where
$\Lae$ is the enveloping algebra $\La \otimes_k \La^{\op}$ of
$\La$). For an Artin $k$-algebra $\Gamma$ and any pair of
$\La$-$\Gamma$-bimodules $X$ and $Y$ the tensor map
$$- \otimes_{\La} X \colon \Ext_{\Lae}^*(\La,\La) \to
\Ext_{\La \otimes_k \Gamma^{\op}}^*(X,X)$$ is a homomorphism of
graded $k$-algebras, making $\Ext_{\La \otimes_k
\Gamma^{\op}}^*(X,Y)$ a left and right
$\Ext_{\Lae}^*(\La,\La)$-module via the maps $- \otimes_{\La} Y$ and
$- \otimes_{\La} X$, respectively (followed by Yoneda product). Now
since $\La$ is projective as a $k$-module, it follows from
\cite[Proposition 3]{Yoneda} that for any two $\La$-$\La$ bimodules
$B$ and $B'$ which are both projective as right $\La$-modules, and
for any homogeneous elements $\eta \in \Ext_{\Lae}^*(B,B')$ and
$\theta \in \Ext_{\La \otimes_k \Gamma^{\op}}^*(X,Y)$, the Yoneda
relation
\begin{equation*}\label{Yonedaext}
(\eta \otimes_{\La} Y) \circ (B \otimes_{\La} \theta ) =
(-1)^{|\eta||\theta|} (B' \otimes_{\La} \theta ) \circ ( \eta
\otimes_{\La} X) \tag{$\dagger$}
\end{equation*}
holds (see also \cite[Theorem 1.1]{Snashall}). By specializing to
the case $\Gamma = \La$ and $B=B'=X=Y= \La$ we see that the
Hochshild cohomology ring $\HH^* ( \La, \La )$ of $\La$ is graded
commutative, whereas the case $\Gamma =k$ and $B=B'= \La$ shows that
for any pair of $\La$-modules $X$ and $Y$ the left and right scalar
actions from $\HH^* ( \La, \La )$ on $\Ext_{\La}^*(X,Y)$ are related
graded commutatively.

Denote the commutative even subalgebra $\bigoplus_{n=0}^{\infty}
\Ext_{\Lae}^{2n}(\La,\La)$ of $\HH^* ( \La, \La )$ by $H^{\ev}$. The
\emph{support variety} of the $\La$-module $M$ is the subset
$$\V(M) \stackrel{\text{def}}{=} \{ \m \in \Maxspec H^{\ev} \mid
\Ann_{H^{\ev}} \Ext_{\La}^*(M,M) \subseteq \m \}$$ of the maximal
ideal spectrum of $H^{\ev}$. As shown in \cite{Erdmann1} and
\cite{Snashall}, the theory of support varieties is rich and in many
ways similar to the theory of cohomological varieties for groups,
especially under the hypothesis that $H^{\ev}$ is Noetherian and
that $\Ext_{\La}^*(X,Y)$ is a finitely generated $H^{\ev}$-module
for all $\La$-modules $X$ and $Y$.

In order to obtain a partly similar theory of twisted support
varieties, the underlying geometric object we use is a ``twisted"
version of the Hochschild cohomology ring. Let $\Gamma$ be any ring
and let $\rho, \phi \colon \Gamma \to \Gamma$ be two ring
automorphisms. If $X$ is a left $\Gamma$-module and $B$ is a
$\Gamma$-$\Gamma$-bimodule, denote by $_{\phi}X$ and
$_{\phi}B_{\rho}$ the left module and bimodule whose scalar actions
are defined by $\gamma \cdot x = \phi ( \gamma )x$ and $\gamma_1
\cdot b \cdot \gamma_2 = \phi ( \gamma_1 )b \rho ( \gamma_2 )$ for
$\gamma, \gamma_1, \gamma_2 \in \Gamma$, $x \in X$, $b \in B$. The
functor assigning to each $\Gamma$-module $X$ the twisted module
$_{\rho}X$ is exact and isomorphic to the functor $_{\rho}\Gamma_1
\otimes_{\Gamma} -$, and it preserves projective modules and minimal
resolutions when the latter makes sense.

Fix a $k$-algebra automorphism
$$\psi \colon \La \to \La$$
and a positive integer $t \in \mathbb{N}$, and consider the graded
$k$-module
$$\HH^{t*}(_{\psi^*}\La_1, \La ) =
\Ext_{\Lae}^{t*}(_{\psi^*}\La_1, \La ) = \bigoplus_{n=0}^{\infty}
\Ext_{\Lae}^{tn}(_{\psi^n}\La_1, \La ).$$ By the above and Lemma
\ref{ring} the pairing
\begin{eqnarray*}
\Ext_{\Lae}^{tm}(_{\psi^m}\La_1, \La ) \times
\Ext_{\Lae}^{tn}(_{\psi^n}\La_1, \La ) & \to &
\Ext_{\Lae}^{t(m+n)}(_{\psi^{m+n}}\La_1, \La ) \\
( \eta, \theta ) & \mapsto & \eta \circ _{\psi^m}\theta
\end{eqnarray*}
defines a multiplication under which $\HH^{t*}(_{\psi^*}\La_1, \La
)$ becomes a graded $k$-algebra, i.e.\ if $\eta$ and $\theta$ are
homogeneous elements given as exact sequences
\begin{eqnarray*}
\eta \colon 0 \to \La \to E_{tm-1} \to \cdots \to E_0 \to
{_{\psi^m}\La_1} \to 0 \\
\theta \colon 0 \to \La \to T_{tn-1} \to \cdots \to T_0 \to
{_{\psi^n}\La_1} \to 0,
\end{eqnarray*}
then their product $\eta \theta$ is given as the Yoneda product of
the sequence $\eta$ with the sequence
$$_{\psi^m}\theta \colon 0 \to {_{\psi^m}\La_1} \to
{_{\psi^m}(T_{tn-1})_1} \to \cdots \to {_{\psi^m}(T_0)_1} \to
{_{\psi^{m+n}}\La_1} \to 0.$$

Before giving an example, recall that a finite dimensional algebra
$\Gamma$ over a field $K$ is \emph{Frobenius} if $_{\Gamma}\Gamma$
and $D(\Gamma_{\Gamma})$ are isomorphic as left $\Gamma$-modules,
where $D$ denotes the usual $K$-dual $\Hom_K(-,K)$. For such an
algebra $\Gamma$, let $\varphi_l \colon _{\Gamma}\Gamma \to
D(\Gamma_{\Gamma})$ be an isomorphism. Let $y \in \Gamma$ be any
element, and consider the linear functional $\varphi_l(1) \cdot y
\in D(\Gamma)$, i.e.\ the $K$-linear map $\Gamma \to K$ defined by
$\gamma \mapsto \varphi_l(1)(y \gamma)$. Since $\varphi_l$ is
surjective there is an element $x \in \Gamma$ having the property
that $\varphi_l(x) = \varphi_l(1) \cdot y$, giving $x \cdot
\varphi_l(1) = \varphi_l(1) \cdot y$ since $\varphi_l$ is a map of
left $\Gamma$-modules. It is not difficult to show that the map $y
\mapsto x$ defines a $K$-algebra automorphism on $\Gamma$, and its
inverse $\nu_{\Gamma}$ is called the \emph{Nakayama automorphism} of
$\Gamma$ (with respect to $\varphi_l$). Thus $\nu_{\Gamma}$ is
defined by $\varphi_l(1)(\gamma x) = \varphi_l(1)( \nu_{\Gamma}(x)
\gamma)$ for all $\gamma \in \Gamma$. This automorphism is unique up
to an inner automorphism; if $\varphi_l' \colon _{\Gamma}\Gamma \to
D(\Gamma_{\Gamma})$ is another isomorphism of left modules yielding
a Nakayama automorphism $\nu_{\Gamma}'$, then there exists an
invertible element $z \in \Gamma$ such that $\nu_{\Gamma} = z
\nu_{\Gamma}' z^{-1}$. Note that $\varphi_l$ is an isomorphism
between the bimodules $_1\Gamma_{\nu_{\Gamma}^{-1}}$ and
$D(\Gamma)$, and since $\nu_{\Gamma}^{-1} \colon
_{\nu_{\Gamma}}\Gamma_1 \to _1\Gamma_{\nu_{\Gamma}^{-1}}$ is an
isomorphism of bimodules we see that $_{\nu_{\Gamma}}\Gamma_1 \simeq
D(\Gamma)$.

As $D(\Gamma_{\Gamma})$ is an injective left $\Gamma$-module, we see
that a Frobenuis algebra is always left selfinjective, but in fact
the definition is left-right symmetric. For if $\varphi_l \colon
_{\Gamma}\Gamma \to D(\Gamma_{\Gamma})$ is an isomorphism of left
$\Gamma$-modules, we can dualize and get an isomorphism $D(
\varphi_l ) \colon D^2( \Gamma_{\Gamma} ) \to D( _{\Gamma}\Gamma )$
of right modules, and composing with the natural isomorphism
$\Gamma_{\Gamma} \simeq D^2( \Gamma_{\Gamma} )$ we get an
isomorphism $\varphi_r \colon \Gamma_{\Gamma} \to D(
_{\Gamma}\Gamma)$ of right $\Gamma$-modules. Moreover, we can view
this last map as an isomorphism $\varphi_r \colon
_{\Gamma^{\op}}\Gamma^{\op} \to D( \Gamma^{\op}_{\Gamma^{\op}} )$ of
left $\Gamma^{\op}$-modules, thereby giving an isomorphism
$\varphi_l \otimes \varphi_r \colon \Gamma \otimes_K \Gamma^{\op}
\to D(\Gamma_{\Gamma}) \otimes_K D( \Gamma^{\op}_{\Gamma^{\op}} )$
of $\Gamma$-$\Gamma$-bimodules. As $D(\Gamma_{\Gamma}) \otimes_K D(
\Gamma^{\op}_{\Gamma^{\op}} )$ is isomorphic to $D( \Gamma \otimes_K
\Gamma^{\op} )$ as a $\Gamma$-$\Gamma$-bimodule, we see that the
enveloping algebra $\Gamma^{\e} = \Gamma \otimes_K \Gamma^{\op}$ of
$\Gamma$ is also Frobenius.

\begin{example}
Suppose $k$ is a field and $\La$ is Frobenius. Recall from the
example in Section \ref{pre} that for a bimodule $B$ in the stable
category $\stmod \Lae$ the orbit algebra
$$\mathbb{A}(B, \tau_{\Lae}) = \Hom_{\Lae} (B,B) \oplus
\bigoplus_{i=1}^{\infty} \stHom_{\Lae} ( \tau_{\Lae}^i(B),B)$$ is
isomorphic to the algebra
$$\bigoplus_{i=0}^{\infty} \Ext_{\Lae}^{2i} ( \mathcal{N}^iB,B),$$
where $\mathcal{N}$ is the Nakayama functor $D \Hom_{\Lae}(-,
\Lae)$. This functor maps the projective cover of a simple
$\Lae$-module to its injective envelope, and is therefore isomorphic
as a functor to $D(\Lae) \otimes_{\Lae} -$. As $\La$ is a Frobenius
algebra, so is $\Lae$, and we know that $D(\Lae)$ is isomorphic to
$_{\nu_{\Lae}}\Lae_1$ as a $\Lae$-$\Lae$-bimodule, giving an
isomorphism
$$\mathbb{A}(B, \tau_{\Lae}) \simeq \bigoplus_{i=0}^{\infty}
\Ext_{\Lae}^{2i} ( _{\nu_{\Lae}^i}B,B)$$ of graded $k$-algebras.
\end{example}

Next we address the question of commutativity in
$\HH^{t*}(_{\psi^*}\La_1, \La )$. For two homogeneous elements
\begin{eqnarray*} \eta \colon 0
\to \La \to E_{tm-1} \to \cdots \to E_0 \to
{_{\psi^m}\La_1} \to 0 \\
\theta \colon 0 \to \La \to T_{tn-1} \to \cdots \to T_0 \to
{_{\psi^n}\La_1} \to 0
\end{eqnarray*}
the Yoneda relation (\ref{Yonedaext}) gives the equality
$$\eta \circ ( _{\psi^m}\La_1 \otimes_{\La} \theta ) =
(-1)^{|\eta||\theta|} \theta \circ ( \eta \otimes_{\La}
{_{\psi^n}\La_1} ),$$ in which the left hand side is the product
$\eta \theta$ (we can identify $_{\psi^m}\La_1 \otimes_{\La} \theta$
with $_{\psi^m}\theta$). However, the extension on the right hand
side is in general \emph{not} the product $\theta \eta$; it is the
Yoneda product of $\theta$ with $\eta_{\psi^{-n}}$, and although we
have bimodule isomorphisms $_1\La_{\psi^{-n}} \simeq
{_{\psi^n}\La_1}$ and $_{\psi^m}\La_{\psi^{-n}} \simeq
{_{\psi^{m+n}}\La_1}$ we cannot in general identify
$\eta_{\psi^{-n}}$ with $_{\psi^n}\eta$ unless $n=0$. In fact, the
following example shows that the algebra $\HH^{t*}(_{\psi^*}\La_1, \La
)$ in general is not graded commutative.

\begin{example}
Suppose $k$ is an algebraically closed field not of characteristic
$2$, and for any natural
numbers $m$ and $n$, with $n \geq 2$, let $Q_{m,n}$ be the quiver
$$
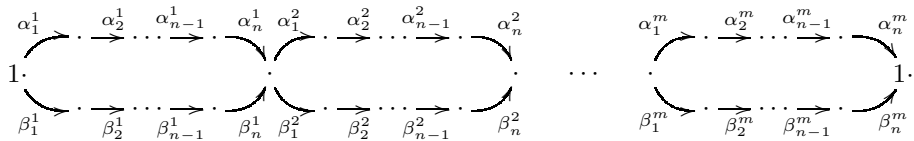
\xymatrix@R=2pt@C=12pt{
& \cdot \ar[r]^{\alpha^1_2} & \cdots \ar[r]^{\alpha^1_{n-1}} & \cdot
\ar@/^/[rd]^<<{\alpha^1_n} && \cdot \ar[r]^{\alpha^2_2} & \cdots
\ar[r]^{\alpha^2_{n-1}} & \cdot \ar@/^/[rd]^{\alpha^2_n} &&&& \cdot
\ar[r]^{\alpha^m_2} & \cdots \ar[r]^{\alpha^m_{n-1}} & \cdot
\ar@/^/[rd]^{\alpha^m_n} \\
  1 \cdot \ar@/^/[ru]^{\alpha^1_1} \ar@/_/[rd]_{\beta^1_1} &&&& \cdot
\ar@/^/[ru]^>>{\alpha^2_1} \ar@/_/[rd]_>>{\beta^2_1} &&&& \cdot &
\cdots & \cdot \ar@/^/[ru]^{\alpha^m_1}
\ar@/_/[rd]_{\beta^m_1} &&&& 1 \cdot \\
& \cdot \ar[r]_{\beta^1_2} & \cdots \ar[r]_{\beta^1_{n-1}} & \cdot
\ar@/_/[ru]_<<{\beta^1_n} && \cdot \ar[r]_{\beta^2_2} & \cdots
\ar[r]_{\beta^2_{n-1}} & \cdot \ar@/_/[ru]_{\beta^2_n} &&&& \cdot
\ar[r]_{\beta^m_2} & \cdots \ar[r]_{\beta^m_{n-1}} & \cdot
\ar@/_/[ru]_{\beta^m_n} }$$
in which the leftmost vertex is the same as the rightmost
vertex. Denote by $\az_{m,n}$ the ideal in the path algebra $kQ_{m,n}$
generated by the paths of length $n+1$, the sums of the parallel paths
of length $n$, and the paths $\{ \beta^1_1 \alpha^m_n, \alpha^1_1
\beta^m_n, \beta^{i+1}_1 \alpha^i_n, \alpha^{i+1}_1 \beta^i_n
\}_{i=1}^{m-1}$. Finally, denote the finite dimensional algebra
$kQ_{m,n} / \az_{m,n}$ by $\La_{m,n}$. It was shown in
\cite{Riedtmann} and \cite{ScW} that the algebras $\left \{ \La_{m,n}
\right \}_{1 \leq m, 2 \leq n}$ represent the distinct stable
equivalence classes of M{\"o}bius algebras. Moreover, in
\cite{Asashiba} and \cite{M-H} it was shown that these algebras also
represent the distinct derived equivalent classes of M{\"o}bius
algebras.

Let $\La$ be an algebra of the above type. By \cite[Theorem 3.5]{EHS}
there is an automorphism $\La \xrightarrow{\psi} \La$ such that
$\Omega_{\Lae}^p( \La ) \simeq {_{\psi^{-1}}\La_1}$, where
$p=2n-1$. Thus the initial part of the minimal bimodule resolution of
$\La$ has the form
$$0 \to {_{\psi^{-1}}\La_1} \to Q_{p-1} \to \cdots \to Q_0 \to \La \to
0,$$
and twisting this exact sequence with $\psi$ we obtain the initial
part
$$0 \to \La \to {_{\psi}(Q_{p-1})_1} \to \cdots \to
{_{\psi}(Q_0)_1} \to {_{\psi}\La_1} \to 0$$
of the minimal bimodule resolution of ${_{\psi}\La_1}$. The latter
exact sequence corresponds to an element $\mu$ in $\HH^{p*}(_{\psi^*}\La_1, \La
)$, and if this algebra is graded commutative then the relation
$\mu^2 = (-1)^{p^2} \mu^2$ implies $\mu^2 =0$ (since $p$ is
odd). However, since $\Omega_{\Lae}^{2p}( {_{\psi^2}\La_1} ) = \La$,
we may represent the element $\mu^2 \in \Ext_{\Lae}^{2p}(
{_{\psi^2}\La_1}, \La)$ by the identity map $\Omega_{\Lae}^{2p}(
{_{\psi^2}\La_1} ) \xrightarrow{\text{id}} \La$, and so if $\mu^2
=0$ then the identity on $\La$ factors through a projective bimodule. Then
$\La$ itself is projective as a bimodule, and therefore semisimple as
an algebra. This is obviously impossible, hence the algebra
$\HH^{p*}(_{\psi^*}\La_1, \La )$ cannot be graded commutative.
\end{example}

The above example shows that the twisted Hochschild
cohomology ring in general is not
graded commutative. In light of this example and the Yoneda relation
(\ref{Yonedaext}), all we can say about
commutativity in $\HH^{t*}(_{\psi^*}\La_1, \La )$ is that a
homogeneous degree zero element commutes with everything, which is
not unexpected since $\HH^0(_{\psi^0}\La_1, \La ) = \Ext_{\Lae}^0(
\La, \La ) = \Z ( \La )$ (the center of $\La$). The next result
gives a criterion under which commutativity relations hold, but
first recall the following from \cite[page 174-175]{Cartan}: for
each $n \geq 0$ let $Q_n$ denote the $n$-fold tensor product of
$\La$ over $k$ (with $Q_0 =k$), and define
$$B^n = Q_{n+2} = \underbrace{\La \otimes_k \cdots \otimes_k
\La}_{n+2 \text{ copies of } \La}.$$ We give $B^n$ a $\Lae$-module
structure by defining
$$\lambda ( \lambda_0 \otimes \cdots \otimes \lambda_{n+1} )
\lambda' = \lambda \lambda_0 \otimes \cdots \otimes \lambda_{n+1}
\lambda',$$ and as $A^{\e}$-modules we then have an isomorphism $B^n
\simeq \Lae \otimes_k Q_n$. Note that $B^n$ is $\Lae$-projective
since the functor $\Hom_{\Lae} ( \Lae \otimes_k Q_n, - )$ is
naturally isomorphic to $\Hom_k ( Q_n, - )$ by adjointness. Now for
each $n \geq 1$, define $d \colon B^n \to B^{n-1}$ by
$$\lambda_0 \otimes \cdots \otimes \lambda_{n+1} \mapsto \sum_{i=0}^n
(-1)^i \lambda_0 \otimes \cdots \otimes \lambda_i \lambda_{i+1}
\otimes \cdots \otimes \lambda_{n+1}.$$ The sequence
$$\mathbb{B}: \cdots \to B^3 \xrightarrow{d} B^2
\xrightarrow{d} B^1 \xrightarrow{d} B^0 \xrightarrow{\mu} \La \to
0,$$ where $\mu$ is the ``multiplication map", is an exact
$\Lae$-projective resolution called the \emph{standard resolution}
(or \emph{Bar-resolution}) of $\La$.

\begin{proposition}\label{commutativity}
Let $\eta \in \HH^{t*}(_{\psi^*}\La_1, \La )$ be a homogeneous
element of degree $tm$ represented by the map $f_{\eta} \colon
_{\psi^m}B^{tm}_1 \to \La$. If for some $n \geq 1$
$$f_{\eta} ( \psi^{-n}( \lambda_0 ) \otimes \cdots \otimes
\psi^{-n}( \lambda_{tm} ) \otimes 1) = \psi^{-n}f_{\eta} ( \lambda_0
\otimes \cdots \otimes \lambda_{tm} \otimes 1)$$ for all $\lambda_0
\otimes \cdots \otimes \lambda_{tm} \otimes 1$ in
$_{\psi^m}B^{tm}_1$, then $\eta_{\psi^{-n}} = {_{\psi^n}\eta}$, and
consequently $\eta \theta = (-1)^{|\eta||\theta|} \theta \eta$ for
every homogeneous element $\theta \in \HH^{t*}(_{\psi^*}\La_1, \La
)$ such that $tn$ divides $|\theta|$. In particular, if
$$f_{\eta} ( \psi^{-1}( \lambda_0 ) \otimes \cdots \otimes
\psi^{-1}( \lambda_{tm} ) \otimes 1) = \psi^{-1}f_{\eta} ( \lambda_0
\otimes \cdots \otimes \lambda_{tm} \otimes 1)$$ for all $\lambda_0
\otimes \cdots \otimes \lambda_{tm} \otimes 1$ in
$_{\psi^m}B^{tm}_1$, then $\eta$ belongs to the graded center of
$\HH^{t*}(_{\psi^*}\La_1, \La )$.
\end{proposition}

\begin{proof}
Suppose the given condition on $f_{\eta}$ holds for $n$. When
viewing $\eta_{\psi^{-n}}$ as a $tm$-fold extension of
${_{\psi^n}\La_1}$ by ${_{\psi^{m+n}}\La_1}$, we use the bimodule
isomorphisms $\psi^{-n} \colon {_{\psi^{m+n}}\La_1} \to
_{\psi^m}\La_{\psi^{-n}}$ and $\psi^{n} \colon _1\La_{\psi^{-n}} \to
_{\psi^n}\La_1$. A lifting of the former along
$_{\psi^{m+n}}\mathbb{B}_1$ is given by the maps
\begin{eqnarray*}
(\psi^{-n})^{\otimes i+2} \colon _{\psi^{m+n}}B^i_1 & \to &
_{\psi^m}B^i_{\psi^{-n}} \\
\lambda_0 \otimes \cdots \otimes \lambda_{i+1} & \mapsto & \psi^{-n}
( \lambda_0) \otimes \cdots \otimes \psi^{-n} ( \lambda_{i+1})
\end{eqnarray*}
for $i \geq 0$, giving the commutative diagram
$$\xymatrix{
\cdots \ar[r] & _{\psi^{m+n}}B^{tm}_1 \ar[d]^{(\psi^{-n})^{\otimes
tm+2}} \ar[r]^{d} & _{\psi^{m+n}}B^{tm-1}_1
\ar[d]^{(\psi^{-n})^{\otimes tm+1}} \ar[r]^{d} & \cdots \ar[r] &
{_{\psi^{m+n}}\La_1} \ar[d]^{\psi^{-n}} \ar[r] & 0 \\
\cdots \ar[r] & _{\psi^m}B^{tm}_{\psi^{-n}} \ar[r]^d
\ar[d]^{f_{\eta}} & _{\psi^m}B^{tm-1}_{\psi^{-n}} \ar[r]^d & \cdots
\ar[r] & _{\psi^m}\La_{\psi^{-n}}
\ar[r] & 0 \\
& _1\La_{\psi^{-n}} \ar[d]^{\psi^n} \\
& _{\psi^n}\La_1 }$$ with exact rows. Therefore the extension
$\eta_{\psi^{-n}}$ is represented by the map composite map
$$\psi^n \circ f_{\eta} \circ
(\psi^{-n})^{\otimes tm+2} \colon _{\psi^{m+n}}B^{tm}_1 \to
_{\psi^n}\La_1,$$ under which the image of an element $\lambda_0
\otimes \cdots \otimes \lambda_{tm+1} \in {_{\psi^{m+n}}B^{tm}_1}$
is easily seen to be $f_{\eta} ( \lambda_0 \otimes \cdots \otimes
\lambda_{tm+1} )$ because of the assumption on $f_{\eta}$.

This shows that we may identify the extensions $\eta_{\psi^{-n}}$
and ${_{\psi^n}\eta}$, and by induction we see that
$\eta_{\psi^{-jn}} = {_{\psi^{jn}}\eta}$ for all $j \geq 1$.
Consequently, if $\theta \in \HH^{t*}(_{\psi^*}\La_1, \La )$ is a
homogeneous element such that $tn$ divides $|\theta|$, say $|\theta|
= tnj$, then the Yoneda relation gives
$$\eta \theta = (-1)^{|\eta||\theta|} \theta \circ
\eta_{\psi^{-jn}} = (-1)^{|\eta||\theta|} \theta \circ
{_{\psi^{jn}}\eta} = (-1)^{|\eta||\theta|} \theta \eta.$$
\end{proof}

For reasons to be explained in the next section, we make the
following definition motivated by the above proposition.

\begin{definition}
A commutative graded subalgebra $H \subseteq
\HH^{t*}(_{\psi^*}\La_1, \La )$ is \emph{strongly commutative} if
$\eta_{\psi^{-1}} = {_{\psi}\eta}$ for all homogeneous elements
$\eta \in H$.
\end{definition}

We end this section with an example of a Frobenius algebra for which
the twisted Hochschild cohomology ring with respect to the Nakayama
automorphism and $t=2$ is strongly commutative.

\begin{example}
Let $k$ be a field not of characteristic $2$ and $q \in k$ a nonzero
element which is \emph{not} a root of unity, and denote by $\La$ the
$k$-algebra
$$\La = k \langle x,y \rangle / ( x^2, xy+qyx, y^2 ).$$
Denote by $D$ the usual $k$-dual $\Hom_k(-,k)$, and consider the map
$\varphi \colon _{\La}\La \to D( \La_{\La} )$ of left $\La$-modules
defined by
$$\varphi(1)( \alpha + \beta x + \gamma y + \delta yx )
\stackrel{\text{def}}{=} \delta.$$ It is easy to show that this is
an injective map and hence also an isomorphism since $\dim_k \La =
\dim_k D( \La )$, and therefore $\La$ is a Frobenius algebra by
definition. Straightforward calculations show that $x \cdot
\varphi(1) = \varphi(1) \cdot (-q^{-1}x)$ and $y \cdot \varphi(1) =
\varphi(1) \cdot (-qy)$, hence since $x$ and $y$ generate $\La$ over
$k$ we see that the Nakayama automorphism $\nu$ (with respect to
$\varphi$) is the degree preserving map defined by
$$x \mapsto -q^{-1}x, \hspace{3mm} y \mapsto -qy.$$

In \cite{Bergh2} the Hochschild cohomology of this $4$-dimensional
graded Koszul algebra was studied, and for every degree preserving
$k$-algebra automorphism $\psi \colon \La \to \La$ the cohomology
groups $\HH^* ( \La , {_1 \Lambda_{\psi}} ) = \Ext_{\Lae}^*( \La ,
{_1 \Lambda_{\psi}} )$ were computed. In particular, from
\cite[3.4(iv)]{Bergh2} we get
$$\dim_k \Ext_{\Lae}^{2n}( \La ,
{_1\Lambda_{\nu^n}} ) = \left \{ \begin{array}{ll}
                                    1 & \text{for } n \text{ even} \\
                                    0 & \text{for } n \text{ odd}
                                  \end{array} \right. $$
when $n>0$, whereas $\Ext_{\Lae}^0( \La , {_1\Lambda_{\nu^0}} ) = \Z
( \La )$ is two dimensional (the elements $1$ and $yx$ form a basis
for the center of $\La$). Moreover, it is not difficult to see that
$\Ext_{\Lae}^{2n}( \La , {_1\Lambda_{\nu^n}} )$ is isomorphic to
$\Ext_{\Lae}^{2n}( _{\nu^n}\La_1, \La )$. Therefore the algebra
$\HH^{4*}(_{\nu^{2*}}\La_1, \La ) = \bigoplus_{n=0}^{\infty}
\Ext_{\Lae}^{4n} ( _{\nu^{2n}}\La_1, \La )$ is two dimensional in
degree zero and one dimensional in the positive degrees.

We now recall the construction of the minimal bimodule projective
resolution of $\La$ from \cite{Buchweitz}. Define the elements
$$f^0_0 = 1, \hspace{3mm} f^1_0 = x, \hspace{3mm} f^1_1 = y, $$
$$f^n_{-1} =0= f^n_{n+1} \hspace{1mm} \text{ for each } n \geq 0,$$
and for each $n \geq 2$ define elements $\{f^n_i\}_{i=0}^n \subseteq
\underbrace{ \La \otimes_k \cdots \otimes_k \La}_{n \text{ copies}}$
inductively by
$$f^n_i = f^{n-1}_{i-1} \otimes y + q^i f^{n-1}_i \otimes x.$$
Denote by $F^n$ the $\Lae$-projective module $\bigoplus_{i=0}^n \La
\otimes_k f^n_i \otimes_k \La$, and by $\tilde{f}^n_i$ the element
$1 \otimes f^n_i \otimes 1 \in F^n$ (and $\tilde{f}^0_0 = 1 \otimes
1$). The set $\{ \tilde{f}^n_i \}_{i=0}^n$ generates $F^n$ as a
$\Lae$-module. Now define a map $\delta \colon F^n \to F^{n-1}$ by
$$\tilde{f}^n_i \mapsto \left [ x \tilde{f}^{n-1}_i + (-1)^n q^i
\tilde{f}^{n-1}_i x \right ] + \left [ q^{n-i} y
\tilde{f}^{n-1}_{i-1} + (-1)^n \tilde{f}^{n-1}_{i-1} y \right ].$$
It is shown in \cite{Buchweitz} that
$$( \mathbb{F}, \delta ) \colon \cdots \to F^{n+1}
\xrightarrow{\delta} F^n \xrightarrow{\delta} F^{n-1} \to \cdots$$
is a minimal $\Lae$-projective resolution of $\La$.

For each $m\geq 1$ consider the $\Lae$-linear map
\begin{eqnarray*}
g_{4m} \colon _{\nu^{2m}}F^{4m}_1 & \to & \La \\
\tilde{f}^{4m}_i & \mapsto & \left \{ \begin{array}{ll}
                               1 & \text{for } i=2m \\
                               0 & \text{for } i \neq 2m.
                                  \end{array} \right.
\end{eqnarray*}
The only generators in $_{\nu^{2m}}F^{4m+1}_1$ which can possibly
map to $\tilde{f}^{4m}_{2m}$ under the map $_{\nu^{2m}}F^{4m+1}_1
\xrightarrow{\delta} _{\nu^{2m}}F^{4m}_1$ are
$\tilde{f}^{4m+1}_{2m}$ and $\tilde{f}^{4m+1}_{2m+1}$, and so it
follows from the equalities
\begin{eqnarray*}
g_{4m} \circ \delta ( \tilde{f}^{4m+1}_{2m} ) &=& g_{4m} (
x\tilde{f}^{4m}_{2m} -q^{2m} \tilde{f}^{4m}_{2m}x ) \\
&=& g_{4m} \left ( [q^{2m}x] \cdot \tilde{f}^{4m}_{2m} -
\tilde{f}^{4m}_{2m} \cdot [q^{2m}x] \right ) \\
&=&0
\end{eqnarray*}
and
\begin{eqnarray*}
g_{4m} \circ \delta ( \tilde{f}^{4m+1}_{2m+1} ) &=& g_{4m} (
q^{2m}y\tilde{f}^{4m}_{2m} - \tilde{f}^{4m}_{2m}y ) \\
&=& g_{4m} \left ( y \cdot \tilde{f}^{4m}_{2m} -
\tilde{f}^{4m}_{2m} \cdot y \right ) \\
&=&0
\end{eqnarray*}
that $g_{4m}$ belongs to the kernel of the map $\delta^* \colon
\Hom_{\Lae} ( _{\nu^{2m}}F^{4m}_1, \La ) \to \Hom_{\Lae}(
_{\nu^{2m}}F^{4m+1}_1, \La )$. Moreover, there cannot exist a
$\Lae$-linear map $_{\nu^{2m}}F^{4m-1}_1 \to \La$ making the diagram
$$\xymatrix{
_{\nu^{2m}}F^{4m}_1 \ar[r]^{\delta} \ar[d]^{g_{4m}} &
_{\nu^{2m}}F^{4m-1}_1 \ar[dl] \\
\La }$$ commute, since $\delta$ is a map of graded degree $1$,
whereas the degree of $g_{4m}$ is $0$. This shows that $g_{4m}$
represents a generator for the one dimensional space
$\Ext_{\Lae}^{4m} ( _{\nu^{2m}}\La_1, \La )$. Note that if we
multiply this map with $yx$, then there \emph{does} exist a map
making the above diagram commute, hence whenever we multiply any
homogeneous element of positive degree in
$\HH^{4*}(_{\nu^{2*}}\La_1, \La )$ with the element $yx \in \HH^0 (
_{\nu^0}\La_1, \La)$ we get zero.

For each $0 \leq i \leq 4$ define $\Lae$-linear maps $\bar{g}_{4m+i}
\colon _{\nu^{2(m+1)}}F^{4m+i}_1 \to _{\nu^2}F^i_1$ by
\begin{eqnarray*}
\bar{g}_{4m} \colon \tilde{f}^{4m}_i & \mapsto & \left \{
   \begin{array}{ll}
   \tilde{f}^0_0 & \text{for } i=2m \\
   0 & \text{otherwise}
   \end{array} \right. \\
\bar{g}_{4m+1} \colon \tilde{f}^{4m+1}_i & \mapsto & \left \{
   \begin{array}{ll}
   q^{2m} \tilde{f}^1_0 & \text{for } i=2m \\
   \tilde{f}^1_1 & \text{for } i=2m+1 \\
   0 & \text{otherwise}
   \end{array} \right. \\
\bar{g}_{4m+2} \colon \tilde{f}^{4m+2}_i & \mapsto & \left \{
   \begin{array}{ll}
   q^{4m} \tilde{f}^2_0 & \text{for } i=2m \\
   q^{2m} \tilde{f}^2_1 & \text{for } i=2m+1 \\
   \tilde{f}^2_2 & \text{for } i=2m+2 \\
   0 & \text{otherwise}
   \end{array} \right. \\
\bar{g}_{4m+3} \colon \tilde{f}^{4m+3}_i & \mapsto & \left \{
   \begin{array}{ll}
   q^{4m} \tilde{f}^3_1 & \text{for } i=2m+1 \\
   q^{2m} \tilde{f}^3_2 & \text{for } i=2m+2 \\
   0 & \text{otherwise}
   \end{array} \right. \\
\bar{g}_{4m+4} \colon \tilde{f}^{4m+4}_i & \mapsto & \left \{
   \begin{array}{ll}
   q^{4m} \tilde{f}^4_2 & \text{for } i=2m \\
   0 & \text{otherwise}.
   \end{array} \right.
\end{eqnarray*}
Then for each $1 \leq i \leq 4$ the diagram
$$\xymatrix{
_{\nu^{2(m+1)}}F^{4m+i}_1 \ar[d]^{\bar{g}_{4m+i}} \ar[r]^{\delta} &
_{\nu^{2(m+1)}}F^{4m+i-1}_1
\ar[d]^{\bar{g}_{4m+i-1}} \\
_{\nu^2}F^i_1 \ar[r]^{\delta} & _{\nu^2}F^{i-1}_1  }$$ commutes,
hence the maps $\bar{g}_{4m+i}$ for $0 \leq i \leq 4$ provide a
lifting of $_{\nu^{2(m+1)}}F^{4m}_1 \xrightarrow{_{\nu^2}g_{4m}}
{_{\nu^2}\La_1}$ along the projective bimodule resolution
$_{\nu^{2(m+1)}}\mathbb{F}_1$ of $_{\nu^{2(m+1)}}\La_1$. Moreover,
the composition $g_4 \circ \bar{g}_{4m+4}$ equals the map
$q^{4m}g_{4m+4}$. Therefore, if we denote by $\theta$ the extension
in $\Ext_{\Lae}^4 ( _{\nu^2}\La_1, \La )$ corresponding to the map
$g_4$, we see that for each $m \geq 2$ the element $\theta^m$ is the
(nonzero) extension in $\Ext_{\Lae}^{4m} ( _{\nu^{2m}}\La_1, \La )$
corresponding to the  map $q^{4(1+2+ \cdots + m-1)}g_{4m}$.

As a result, we see that we have an isomorphism
$$\HH^{4*}(_{\nu^{2*}}\La_1, \La ) \simeq k[x] \times_k k$$
of graded $k$-algebras, with $x$ of degree $4$, and in particular
$\HH^{4*}(_{\nu^{2*}}\La_1, \La )$ is commutative. Moreover, it is
easily shown that we may identify the extensions $\theta_{\nu^{-2}}$
and ${_{\nu^2}\theta}$, and so $\HH^{4*}(_{\nu^{2*}}\La_1, \La )$ is
actually strongly commutative.
\end{example}

\section{Twisted support varieties}\label{var}

In order to obtain a theory of twisted support
varieties, the underlying geometric object we shall use is the
twisted Hochschild cohomology ring $\HH^{t*}(_{\psi^*}\La_1, \La ) =
\Ext_{\Lae}^{t*}(_{\psi^*}\La_1, \La )$ studied in the previous
section (where $\La \xrightarrow{\psi} \La$ is an automorphism and
$t \in \mathbb{N}$). Similarly as for $\HH^{t*}(_{\psi^*}\La_1, \La
)$, the graded $k$-module $\Ext_{\La}^{t*} ( _{\psi^*}M,M) =
\bigoplus_{n=0}^{\infty} \Ext_{\La}^{tn}(_{\psi^n}M,M)$ is a
$k$-algebra, and for each $\La$-module $N$ the graded $k$-module
$\Ext_{\La}^{t*} ( _{\psi^*}M,N) = \bigoplus_{n=0}^{\infty}
\Ext_{\La}^{tn}(_{\psi^n}M,N)$ becomes a graded right
$\Ext_{\La}^{t*} ( _{\psi^*}M,M)$-module by defining
$$\zeta \mu \stackrel{\text{def}}{=} \zeta \circ {_{\psi^m}\mu}$$
for $\zeta \in \Ext_{\La}^{tm}(_{\psi^m}M,N)$ and $\mu \in
\Ext_{\La}^{tn}(_{\psi^n}M,M)$. Moreover, the tensor map
\begin{eqnarray*}
- \otimes_{\La} M \colon \HH^{t*}(_{\psi^*}\La_1, \La ) & \to &
\Ext_{\La}^{t*} ( _{\psi^*}M,M) \\
\eta & \mapsto & \eta \otimes_{\La} M
\end{eqnarray*}
is a homomorphism of graded $k$-algebras, thus making
$\Ext_{\La}^{t*} ( _{\psi^*}M,N)$ a graded right
$\HH^{t*}(_{\psi^*}\La_1, \La )$-module. Similarly $\Ext_{\La}^{t*}
( _{\psi^*}M,N)$ becomes a graded left $\HH^{t*}(_{\psi^*}\La_1, \La
)$ via the homomorphism
\begin{eqnarray*}
- \otimes_{\La} N \colon \HH^{t*}(_{\psi^*}\La_1, \La ) & \to &
\Ext_{\La}^{t*} ( _{\psi^*}N,N) \\
\eta & \mapsto & \eta \otimes_{\La} N
\end{eqnarray*}
of algebras.

We now make the following assumption:

\begin{assumption}
Given the automorphism $\La \xrightarrow{\psi} \La$ and the integer
$t \in \mathbb{N}$, the graded subalgebra $H =
\bigoplus_{n=0}^{\infty} H^{tn}$ of $\HH^{t*}({_{\psi^*}\La_1}, \La
)$ is strongly commutative with $H^0 = \HH^0 ( \La, \La ) = \Z ( \La
)$.
\end{assumption}

Why do we restrict ourselves to strongly commutative algebras
instead of ``ordinary" commutative algebras? The answer is that we
want the bifunction $\Ext_{\La}^{t*}( _{\psi^*}-,-)$, which maps the
pair $(M,N)$ of $\La$-modules to the $H$-module $\Ext_{\La}^{t*}(
{_{\psi^*}M},N)$, to preserve maps. Namely, consider two homogeneous
elements $\zeta \in \Ext_{\La}^{tm}(_{\psi^m}M,N)$ and $\eta \in
\Ext_{\Lae}^{tn}({_{\psi^n}\La_1}, \La )$. By definition $\zeta
\cdot \eta$ is the extension $\zeta \circ {_{\psi^m}( \eta
\otimes_{\La} M )} = \zeta \circ ( {_{\psi^m}\eta} \otimes_{\La} M
)$, and if ${_{\psi^m}\eta} = \eta_{\psi^{-m}}$ we get
${_{\psi^m}\eta} \otimes_{\La} M = \eta_{\psi^{-m}} \otimes_{\La} M
= \eta \otimes_{\La} {_{\psi^m}M}$. Hence in this case we may
identify the right scalar product $\zeta \cdot \eta$ with $( \La
\otimes_{\La} \zeta ) \circ ( \eta \otimes_{\La} {_{\psi^m}M} )$,
and then by the Yoneda relation (\ref{Yonedaext}) from the previous
section we get
$$\zeta \cdot \eta = ( \La
\otimes_{\La} \zeta ) \circ ( \eta \otimes_{\La} {_{\psi^m}M} ) =
(-1)^{|\zeta||\eta|} ( \eta \otimes_{\La} N ) \circ (
{_{\psi^n}\La_1} \otimes_{\La} \zeta ).$$ However, the extension $(
\eta \otimes_{\La} N ) \circ ( {_{\psi^n}\La_1} \otimes_{\La} \zeta
)$ is precisely the left scalar product $\eta \cdot \zeta$, and so
if $H$ is strongly commutative we see that
$$\zeta \cdot \eta = (-1)^{|\zeta||\eta|} \eta \cdot \zeta$$
for all homogeneous elements $\eta \in H$ and $\zeta \in
\Ext_{\La}^{t*}(_{\psi^*}M,N)$.

The fact that the left and right scalar multiplications basically
coincide is what makes $\Ext_{\La}^{t*}( _{\psi^*}-,-)$ preserve
maps. To see this, let $f \colon M \to M'$ be a $\La$-homomorphism.
This map induces a homomorphism $\hat{f} \colon
\Ext_{\La}^{t*}(_{\psi^*}M',N) \to \Ext_{\La}^{t*}({_{\psi^*}M},N)$
of graded groups, under which the image of a homogeneous element
$$\theta \colon 0 \to N \to X_{tn} \to \cdots \to X_1 \to
{_{\psi^n}M'} \to 0$$ is the extension $\theta f$ given by the
commutative diagram
$$\xymatrix{
\theta f \colon 0 \ar[r] & N \ar[r] \ar@{=}[d] & X_{tn} \ar[r]
\ar@{=}[d] & \cdots \ar[r] & X_2 \ar[r] \ar@{=}[d] & Y \ar[r]
\ar[d] & {_{\psi^n}M} \ar[r] \ar[d]^f & 0 \\
\theta \colon 0 \ar[r] & N \ar[r] & X_{tn} \ar[r] & \cdots \ar[r] &
X_2 \ar[r] & X_1 \ar[r] & {_{\psi^n}M'} \ar[r] & 0 }$$ in which the
module $Y$ is a pullback. For a homogeneous element $\eta \in
H^{tm}$ we then get
\begin{eqnarray*}
\hat{f} ( \theta \cdot \eta ) & = & (-1)^{|\eta||\theta|} \hat{f} (
\eta \cdot \theta ) \\
& = & (-1)^{|\eta||\theta|} \hat{f} \left ( ( \eta \otimes_{\La} N )
\circ {_{\psi^m}\theta} \right ) \\
& = & (-1)^{|\eta||\theta|} ( \eta \otimes_{\La} N ) \circ
{_{\psi^m}( \theta f)} \\
& = & (-1)^{|\eta||\theta|} \eta \cdot \hat{f} ( \theta ) \\
& = & \hat{f} ( \theta ) \cdot \eta,
\end{eqnarray*}
showing $\hat{f}$ is a homomorphism of $H$-modules. Similarly
$\Ext_{\La}^{t*}( _{\psi^*}-,-)$ preserves maps in the second
argument.

\begin{remark}
It should also be noted that in many cases considering only
strongly commutative algebras is not a severe restriction. For
example, let $H$ be a commutative Noetherian graded subalgebra of
$\HH^{t*}({_{\psi^*}\La_1}, \La )$, as will be the case in most of
the results in this section. Then by \cite[Proposition 2.1]{Bergh1}
there exists an integer $w$ and a homogeneous element $\eta \in H$
of positive degree, say $|\eta| = tn$, such that the multiplication
map
$$H^{ti} \xrightarrow{\eta \cdot} H^{t(i+n)}$$
is injective for $i \geq w$. Now for any homogeneous element $\theta
\in H^{tm}$ the product $\theta \cdot \eta$ is by definition the
extension $\theta \circ {_{\psi^m}\eta} = ( \theta \otimes_{\La} \La
) \circ ( {_{\psi^m}\La_1} \otimes_{\La} \eta )$, which by the
Yoneda relation (\ref{Yonedaext}) from Section \ref{hoc} equals the
extension $(-1)^{|\eta||\theta|}( \La \otimes_{\La} \eta ) \circ (
\theta \otimes_{\La} {_{\psi^n}\La_1} ) = (-1)^{|\eta||\theta|} \eta
\circ \theta_{\psi^{-n}}$. However, the extension $\eta \circ
\theta_{\psi^{-n}}$ corresponds to the product $\eta \cdot
{_{\psi^{-n}}\theta_{\psi^{-n}}}$ in $H$, and so since $H$ is
commutative we get
$$\eta \cdot \theta = \theta \cdot \eta = (-1)^{|\eta||\theta|} \eta
\cdot {_{\psi^{-n}}\theta_{\psi^{-n}}}.$$ Now if $|\theta| \geq tw$
then $\theta = (-1)^{|\eta||\theta|}
{_{\psi^{-n}}\theta_{\psi^{-n}}}$ due to the injectivity of the
multiplication map induced by $\eta$, and twisting both sides in
this equality by $\psi^n$ from the left we get ${_{\psi^n}\theta} =
(-1)^{|\eta||\theta|} \theta_{\psi^{-n}}$.

The algebra $H$, being Noetherian, is generated over $H^0$ by a
finite set of homogeneous elements, say $\{ x_0, \dots, x_r \}$.
Suppose now that $k$ contains an infinite field. Then from the proof
of \cite[Theorem 2.5]{Bergh1} we see that we may choose the element
$\eta$ above such that $|\eta| = \lcm \{ |x_0|, \dots, |x_r| \}$ In
particular, if all the generators $x_0, \dots, x_r$ are of the same
degree, i.e.\ $|x_i|=t$, then the subalgebra $H^0 \oplus
\bigoplus_{i=w}^{\infty} H^{ti}$ of $H$ is strongly commutative.
\end{remark}

Having made all the necessary assumptions, we are now ready to
define twisted support varieties. For two $\La$-modules $X$ and $Y$,
denote by $\A_H^{\psi} (X,Y)$ the ideal $\Ann_H \Ext_{\La}^{t*} (
_{\psi^*}X, Y )$, i.e.\ the annihilator in $H$ of $\Ext_{\La}^{t*} (
_{\psi^*}X, Y )$. This ideal is graded since $\Ext_{\La}^{t*} (
_{\psi^*}X, Y )$ is a graded $H$-module. As for ordinary support
varieties, the defining ideal for the twisted support variety of $M$
is $\A_H^{\psi}(M,\La / \ra )$.

\begin{definition}\label{twistedvarieties}
The \emph{twisted support variety} $\V_H^{\psi} (M)$ of $M$ (with
respect to the automorphism $\psi$, the integer $t$ and the ring
$H$) is the subset
$$\V_H^{\psi} (M) \stackrel{\text{def}}{=}\{ \m \in \Maxspec H \mid
\A_H^{\psi}(M,\La / \ra ) \subseteq \m \}$$ of the maximal ideal
spectrum of $H$.
\end{definition}

Note that since $M$ and therefore also $\Ext_{\La}^{t*} (
_{\psi^*}M, \La / \ra )$ is nonzero, the degree zero part of
$\A_H^{\psi}(M,\La / \ra )$ must be contained in the radical $\rad
H^0$ of the local ring $H^0= \Z ( \La )$ (if not then
$\A_H^{\psi}(M,\La / \ra )$ contains the identity in $H$). Hence the
twisted variety $\V_H^{\psi} (M)$ always contains the unique graded
maximal ideal
$$\m_H = \rad H^0 \oplus H^{t} \oplus H^{2t} \oplus \cdots$$ of $H$,
and we shall say the variety is \emph{trivial} if it does not
contain any other element.

The following proposition lists some elementary and more or less
expected properties of twisted varieties.

\begin{proposition}\label{properties}
\begin{enumerate}
\item[(i)] The ideal $\A_H^{\psi}(M,M)$ is also a defining ideal for
the twisted variety of $M$, i.e.\
$$\V_H^{\psi} (M) = \{ \m \in \Maxspec H \mid \A_H^{\psi}(M,M)
\subseteq \m \}.$$
\item[(ii)] If $\Ext_{\La}^{tn} ( _{\psi^n}M, M)=0$ for $n \gg 0$,
then $\V_H^{\psi} (M)$ is trivial. In particular, the twisted
variety of $M$ is trivial whenever the projective or injective
dimension of $M$ is finite.
\item[(iii)] Given nonzero $\La$-modules $M'$ and $M''$ and an exact
sequence
$$0 \to M' \to M \to M'' \to 0,$$ the relation
$$\V_H^{\psi} (M) \subseteq \V_H^{\psi} (M') \cup
\V_H^{\psi} (M'')$$ holds. Moreover, if $\Omega_{\La}^{t-1}(M') \neq
0$, then the relation
$$\V_H^{\psi} (M'') \subseteq \V_H^{\psi} (M) \cup \V_H^{\psi} (
\Omega_{\La}^{t-1}( _{\psi}M'))$$ holds, and if $M''=
\Omega_{\La}^{t-1}(M''')$ for some module $M'''$, then the relation
$$\V_H^{\psi} (M') \subseteq \V_H^{\psi} (M) \cup
\V_H^{\psi} ( _{\psi^{-1}}M''')$$ holds.
\item[(iv)] If $M'$ and $M''$ are nonzero $\La$-modules such that
$M=M' \oplus M''$, then $\V_H^{\psi} (M) = \V_H^{\psi} (M') \cup
\V_H^{\psi} (M'')$.
\end{enumerate}
\end{proposition}

\begin{proof}
For a $\La$-module $N$, denote by $\V_H^{\psi} (M,N)$ the variety
whose defining ideal is $\A_H^{\psi}(M,N)$. To prove (i), we must
show that $\V_H^{\psi}(M, \La / \ra )$ and $\V_H^{\psi}(M,M)$ are
equal, and we start by proving by induction on the length of $N$
that $\V_H^{\psi} (M,N)$ is contained in $\V_H^{\psi} (M, \La / \ra
)$. For any simple $\La$-module $S$ the ideal $\A_H^{\psi}(M, \La /
\ra )$ is contained in $\A_H^{\psi}(M,S)$ since $S$ is a summand of
$\La / \ra$, hence $\V_H^{\psi} (M,S) \subseteq \V_H^{\psi}(M, \La /
\ra )$. If the length of $N$ is greater than $2$, take any simple
submodule $S \subset N$ and consider the exact sequence
$$0 \to S \to N \to N/S \to 0.$$
This sequence induces the exact sequence
$$\Ext_{\La}^{t*}( _{\psi^*}M,S) \to \Ext_{\La}^{t*}( _{\psi^*}M,N)
\to \Ext_{\La}^{t*}( _{\psi^*}M,N/S)$$ of $H$-modules, from which
the inclusion
$$\A_H^{\psi}(M,S) \cdot \A_H^{\psi}(M,N/S) \subseteq
\A_H^{\psi}(M,N)$$ follows. This implies that $\V_H^{\psi} (M,N)
\subseteq \V_H^{\psi} (M,S) \cup \V_H^{\psi} (M,N/S)$, and since
both $\V_H^{\psi} (M,S)$ and $\V_H^{\psi} (M,N/S)$ are contained in
$\V_H^{\psi} (M, \La / \ra )$ by induction, we get $\V_H^{\psi}
(M,N) \subseteq \V_H^{\psi} (M, \La / \ra )$. In particular
$\V_H^{\psi} (M,M)$ is contained in $\V_H^{\psi} (M, \La / \ra )$.
The reverse inclusion $\V_H^{\psi} (M, \La / \ra ) \subseteq
\V_H^{\psi} (M,M)$ follows from the inclusion $\A_H^{\psi}(M,M)
\subseteq \A_H^{\psi}(M, \La / \ra )$ of ideals in $H$, thus proving
(i).

Now suppose $\Ext_{\La}^{tn} ( _{\psi^n}M, M)=0$ for $n \gg 0$, let
$\m \subset H$ be a maximal ideal containing $\A_H^{\psi}(M, \La /
\ra )$, and let $x \in H^+ = \bigoplus_{n=1}^{\infty}H^{tn}$ be a
homogeneous element. As the scalar action from $H$ on
$\Ext_{\La}^{t*} ( _{\psi^*}M, \La / \ra )$ factors through
$\Ext_{\La}^{t*} ( _{\psi^*}M, M)$, some power of $x$ must lie in
$\A_H^{\psi}(M, \La / \ra )$. Therefore $x$ must be an element of
$\m$, implying $\m$ is a graded ideal. But $\m_H$ is the only graded
maximal ideal of $H$. This proves (ii).

Next suppose we are given an exact sequence as in (iii). The
sequence induces the exact sequence
$$\Ext_{\La}^{t*}( _{\psi^*}M'', \La / \ra ) \to \Ext_{\La}^{t*}(
_{\psi^*}M, \La / \ra ) \to \Ext_{\La}^{t*}( _{\psi^*}M', \La /
\ra )$$ of $H$-modules, giving the inclusion
$$\A_H^{\psi}(M',\La / \ra ) \cdot \A_H^{\psi}(M'',\La / \ra )
\subseteq \A_H^{\psi}(M,\La / \ra )$$ of ideals of $H$. The relation
$\V_H^{\psi} (M) \subseteq \V_H^{\psi} (M') \cup \V_H^{\psi} (M'')$
now follows.

The original exact sequence also induces the two exact sequences
\begin{eqnarray*}
& 0 \to \Hom_{\La} (M'', \La / \ra ) \to \Hom_{\La} (M, \La / \ra
) & \\
& \Ext_{\La}^{tn-1}( _{\psi^n}M', \La / \ra ) \to \Ext_{\La}^{tn}(
_{\psi^n}M'', \La / \ra ) \to \Ext_{\La}^{tn}( _{\psi^n}M, \La /
\ra ) &
\end{eqnarray*}
for $n \geq 1$, in which we may identify $\Ext_{\La}^{tn-1}(
_{\psi^n}M', \La / \ra )$ with $\Ext_{\La}^{t(n-1)}(
_{\psi^{n-1}}\Omega_{\La}^{t-1}(_{\psi}M'), \La / \ra )$.
Consequently the inclusion
$$\A_H^{\psi}(M,\La / \ra ) \cdot \A_H^{\psi}(\Omega_{\La}^{t-1}(
_{\psi}M'),\La / \ra ) \subseteq \A_H^{\psi}(M'',\La / \ra )$$ holds
whenever $\Omega_{\La}^{t-1}(M') \neq 0$, giving $\V_H^{\psi} (M'')
\subseteq \V_H^{\psi} (M) \cup \V_H^{\psi} ( \Omega_{\La}^{t-1}(
_{\psi}M'))$.

Finally, the short exact sequence induces the exact sequence
$$\Ext_{\La}^{tn}( _{\psi^n}M, \La / \ra ) \to \Ext_{\La}^{tn}(
_{\psi^n}M', \La / \ra ) \to \Ext_{\La}^{tn+1}( _{\psi^n}M'', \La /
\ra )$$ for $n \geq 0$. If $M''= \Omega_{\La}^{t-1}(M''')$ for some
module $M'''$, then we may identify $\Ext_{\La}^{tn+1}(
_{\psi^n}M'', \La / \ra )$ with $\Ext_{\La}^{t(n+1)}(
_{\psi^{n+1}}(_{\psi^{-1}}M'''), \La / \ra )$, and the inclusion
$$\A_H^{\psi}(M,\La / \ra ) \cdot \A_H^{\psi}( _{\psi^{-1}}M''',\La
/ \ra ) \subseteq \A_H^{\psi}(M',\La / \ra )$$ holds. This gives
$\V_H^{\psi} (M') \subseteq \V_H^{\psi} (M) \cup \V_H^{\psi} (
_{\psi^{-1}}M''')$, and the proof of (iii) is complete.

To prove (iv), note that by (iii) the inclusion $\V_H^{\psi} (M)
\subseteq \V_H^{\psi} (M') \cup \V_H^{\psi} (M'')$ holds, whereas
the inclusions $\A_H^{\psi}(M,\La / \ra ) \subseteq
\A_H^{\psi}(M',\La / \ra )$ and $\A_H^{\psi}(M,\La / \ra ) \subseteq
\A_H^{\psi}(M'',\La / \ra )$ give $\V_H^{\psi} (M') \subseteq
\V_H^{\psi} (M)$ and $\V_H^{\psi} (M'') \subseteq \V_H^{\psi} (M)$.
\end{proof}

As a corollary, we obtain a result whose analogue in the theory of
support varieties says that varieties are invariant under the
syzygy operator.

\begin{corollary}\label{syzygy}
If $\Omega_{\La}^t(M)$ is nonzero, then $\V_H^{\psi} (M) =
\V_H^{\psi} ( \Omega_{\La}^t ( _{\psi}M))$.
\end{corollary}

\begin{proof}
The exact sequence
$$0 \to \Omega_{\La}^1(M) \to P_0 \to M \to 0$$
gives $\V_H^{\psi} (M) \subseteq \V_H^{\psi} (P_0) \cup
\V_H^{\psi} ( \Omega_{\La}^{t-1}( _{\psi}\Omega_{\La}^1(M)))$, and
since $\V_H^{\psi} (P_0)$ is trivial we get $\V_H^{\psi} (M)
\subseteq \V_H^{\psi} ( \Omega_{\La}^t ( _{\psi}M))$. On the other
hand, the exact sequence
$$0 \to \Omega_{\La}^t( _{\psi}M) \to
{_{\psi}P_{t-1}} \to \Omega_{\La}^{t-1}( _{\psi}M) \to 0$$ gives
$\V_H^{\psi} (\Omega_{\La}^t( _{\psi}M)) \subseteq \V_H^{\psi}(
_{\psi}P_{t-1}) \cup \V_H^{\psi}( _{\psi^{-1}}( _{\psi}M))$, and
since $\V_H^{\psi}( _{\psi}P_{t-1})$ is trivial we get $\V_H^{\psi}
(\Omega_{\La}^t( _{\psi}M)) \subseteq \V_H^{\psi} (M)$.
\end{proof}

We illustrate this last result with an example.

\begin{example}
Suppose $k$ is a field and $\La$ is a Frobenius algebra, let $\psi =
\nu_{\La}$ be the Nakayama automorphism of $\La$, and take $t=2$. We
saw in the example following Lemma \ref{ring} that the
Auslander-Reiten translation $\tau = D \Tr$ is isomorphic to
$\Omega_{\La}^2 \mathcal{N}$, where $\mathcal{N}$ is the Nakayama
functor $D \Hom_{\La}(-, \La)$. Moreover, in the example prior to
Proposition \ref{commutativity} we saw that the latter is isomorphic
to $_{\nu}\La_1 \otimes_{\La} -$. Therefore, from the corollary
above we get
$$\V_H^{\nu}(M) = \V_H^{\nu}( \Omega_{\La}^2 ( _{\nu}M)) =
\V_H^{\nu}( \Omega_{\La}^2 ( _{\nu}\La_1 \otimes_{\La} M)) =
\V_H^{\nu}( \tau (M))$$ whenever $\Omega_{\La}^2 (M)$ is nonzero.
\end{example}

A fundamental feature within the theory of support varieties for
modules over both group algebras of finite groups and complete
intersections is the finite generation of $\Ext^*(X,Y)$ (where $X$
and $Y$ are modules over the ring in question) as a module over the
commutative Noetherian graded ring of cohomological operators (see
\cite{Benson} and \cite{Carlson} for the group ring case and
\cite{Avramov1} and \cite{Avramov2} for the complete intersection
case). In order to obtain a similar theory for selfinjective
algebras in \cite{Erdmann1}, it was necessary to \emph{assume} that
the same finite generation hypothesis held, since there exist
selfinjective algebras for which it does not hold.

We now introduce a finite generation hypothesis similar to that used
in \cite{Bergh1}, where instead of assuming finite generation of
$\Ext^*(X,Y)$ for \emph{all} modules $X$ and $Y$, a local variant
focusing on a single module was used.

\begin{assumption}[\bf{Fg($M,H,\psi,t$)}]\label{fg}
Given the automorphism $\La \xrightarrow{\psi} \La$ and the integer
$t \in \mathbb{N}$, there exists a strongly commutative Noetherian
graded subalgebra $H = \bigoplus_{n=0}^{\infty} H^{tn}$ of
$\HH^{t*}(_{\psi^*}\La_1, \La )$ such that $H^0 = \HH^0 ( \La, \La )
= \Z ( \La )$, and with the property that $\Ext_{\La}^{t*} (
_{\psi^*}M, \La / \ra )$ is a finitely generated $H$-module.
\end{assumption}

As mentioned in the remark prior to Definition
\ref{twistedvarieties}, considering only strongly commutative
algebras instead of ``ordinary" commutative algebras is not a severe
restriction. In fact, the following result shows that if
{\bf{Fg($M,H,\psi,t$)}} holds for a commutative Noetherian graded
algebra $H$ which is not necessarily strongly commutative, then
there exist a positive integer $s$ and a strongly commutative
Noetherian subalgebra $H' \subseteq H$ for which
{\bf{Fg($M,H',\psi^s,ts$)}} holds.

\begin{proposition}\label{commutativeisstrong}
Let $\La \xrightarrow{\psi} \La$ be an automorphism and $t \in
\mathbb{N}$ an integer, and suppose there exists a commutative
Noetherian graded subalgebra $H = \bigoplus_{n=0}^{\infty} H^{tn}$
of $\HH^{t*}(_{\psi^*}\La_1, \La )$ such that $H^0 = \HH^0 ( \La,
\La ) = \Z ( \La )$, and with the property that $\Ext_{\La}^{t*} (
_{\psi^*}M, \La / \ra )$ is a finitely generated $H$-module. Then
there exist a positive integer $s$ and a strongly commutative
Noetherian graded subalgebra $H'$ of $H$ for which
{\bf{Fg($M,H',\psi^s,ts$)}} holds
\end{proposition}

\begin{proof}
From the remark prior to Definition \ref{twistedvarieties} we see
that there exist two positive integers $n$ and $w$ such that
${_{\psi^n}\theta} = \theta_{\psi^{-n}}$ for every homogeneous
element $\theta \in H$ with $|\theta| \geq tw$. Let $\eta_1, \dots,
\eta_r \in H$ be homogeneous elements of positive degrees generating
$H$ as an algebra over $H^0$, and denote the integer $wn$ by $s$.
Then the subalgebra $H' = H^0[\eta_1^s, \dots, \eta_r^s]$ of $H$ is
strongly commutative with respect to the automorphism $\psi^s$, and
$\Ext_{\La}^{(ts)*} ( _{(\psi^s)^*}M, \La / \ra )$ is a finitely
generated $H'$-module.
\end{proof}

We now show that introducing the above finite generation hypothesis
enables us to compute the dimension of the twisted variety of a
module. Recall first that if $X = \bigoplus_{n=0}^{\infty} X_n$ is a
graded $k$-module of finite type (that is, the $k$-length of $X_n$
is finite for all $n$), then the \emph{rate of growth} of $X$,
denoted $\gamma (X)$, is defined as
$$\gamma (X) \stackrel{\text{def}}{=} \inf \{ c \in \mathbb{N} \cup
\{ 0 \} \mid \exists a \in \mathbb{R} \text{ such that } \ell_k(X_n)
\leq an^{c-1} \text{ for } n \gg 0 \}.$$ For a $\La$-module $Y$ with
minimal projective resolution $\cdots \to Q_1 \to Q_0 \to Y \to 0$,
we define the \emph{$m$-complexity} of $Y$ (where $m \geq 1$ is a
number) as the rate of growth of $\bigoplus_{n=0}^{\infty}Q_{mn}$,
and denote it by $\cx^m Y$. Note that the $1$-complexity of a module
coincides with the usual notion of complexity, and that the identity
$$\cx^m Y = \gamma \left ( \Ext_{\La}^{m*} (Y, \La / \ra ) \right )$$
always holds (the latter can be seen by adopting the arguments given
in \cite[Section 3]{Bergh1}).

The following result allows us to compute the dimension of
$\V_H^{\psi} (M)$ in terms of the $t$-complexity of $M$ provided
{\bf{Fg($M,H,\psi,t$)}} is satisfied. In particular Dade's Lemma
holds.

\begin{proposition}\label{dimension}
If \emph{\bf{Fg($M,H,\psi,t$)}} holds, then $\dim \V_H^{\psi} (M) =
\cx^t M$. In particular $\V_H^{\psi} (M)$ is trivial if and only if
$M$ has finite projective dimension.
\end{proposition}

\begin{proof}
Adopting the arguments used to prove \cite[Proposition
5.7.2]{Benson} and \cite[Proposition 2.1]{Erdmann1} gives $\dim
\V_H^{\psi} (M) = \gamma \left ( \Ext_{\La}^{t*} ( _{\psi^*}M, \La /
\ra ) \right )$. For any $\La$-module $Y$ and any $k$-automorphism
$\La \xrightarrow{\phi} \La$ there is an isomorphism $Y \simeq
{_{\phi}Y}$ of $k$-modules, and so if $\cdots \to Q_1 \to Q_0 \to Y
\to 0$ is a minimal projective resolution we see that the
$m$-complexity of $Y$ equals the rate of growth of
$\bigoplus_{n=0}^{\infty} {_{\phi^n}(Q_{mn})}$. In particular the
equalities
$$\cx^t M = \gamma \left ( \bigoplus_{n=0}^{\infty}
{_{\psi^n}(P_{tn})} \right ) = \gamma \left ( \Ext_{\La}^{t*} (
_{\psi^*}M, \La / \ra ) \right )$$ hold, where
$$\mathbb{P} \colon \cdots \to P_2 \xrightarrow{d_2} P_1
\xrightarrow{d_1} P_0 \xrightarrow{d_0} M \to 0$$ is the minimal
projective resolution of $M$.
\end{proof}

Note that whenever {\bf{Fg($M,H,\psi,t$)}} is satisfied the
dimension of $\V_H^{\psi} (M)$, and therefore also the
$t$-complexity of $M$, must be finite; since $H$ is commutative
graded Noetherian it is generated as an algebra over $H^0$ by a
finite set $\{ x_0, \dots, x_r \}$ of homogeneous elements of
positive degrees, giving $\gamma \left ( \Ext_{\La}^{t*} (
_{\psi^*}M, \La / \ra ) \right ) \leq r$ (see the discussion prior
to \cite[Proposition 3.1]{Bergh1}). It then follows from the proof
of the above proposition that $\cx^tM \leq r$.

The next result gives a sufficient and necessary condition for the
variety to be one dimensional. Recall that a $\La$-module $Y$ is
\emph{periodic} if there exists a positive integer $p$ such that $Y
\simeq \Omega_{\La}^p(Y)$, whereas it is \emph{eventually periodic}
if $\Omega_{\La}^i(Y)$ is periodic for some $i \geq 0$. For a
$k$-automorphism $\La \xrightarrow{\phi} \La$ we define $Y$ to be
\emph{$\phi$-periodic} if there exists a positive integer $p$ such
that $Y \simeq \Omega_{\La}^p(_{\phi}Y)$, and \emph{eventually
$\phi$-periodic} if $\Omega_{\La}^i(Y)$ is $\phi$-periodic for some
$i \geq 0$.

\begin{proposition}\label{periodic}
If \emph{\bf{Fg($M,H,\psi,t$)}} holds, then $\dim \V_H^{\psi} (M)
=1$ if and only if $M$ is eventually $\psi^i$-periodic for some $i
\geq 1$. Moreover, when this occurs there is a positive integer
$w$ such that $\Omega_{\La}^{tj}(M) \simeq \Omega_{\La}^{t(j+w)}(
_{\psi^w}M)$ for some $j \geq 0$.
\end{proposition}

\begin{proof}
If $M$ is eventually $\psi^i$-periodic, then the sequence
$$\ell_k (P_0), \ell_k (P_1), \ell_k (P_2), \dots$$
must be bounded, that is, the $1$-complexity of $M$ is $1$. But then
the sequence
$$\ell_k (P_0), \ell_k (P_t), \ell_k (P_{2t}), \dots$$
is also bounded, that is, the $t$-complexity of $M$ is also $1$, and
consequently $\dim \V_H^{\psi} (M) =1$.

Conversely, suppose the latter of the above sequences is bounded. By
\cite[Proposition 2.1]{Bergh1} there exists a homogeneous element
$\eta \in H$ of positive degree, say $|\eta| = tw$, such that the
multiplication map
$$\Ext_{\La}^{ti} ( _{\psi^i}M, \La / \ra ) \xrightarrow{\cdot \eta}
\Ext_{\La}^{t(i+w)} ( _{\psi^{i+w}}M, \La / \ra )$$ is injective for
$i \gg 0$. Represent the element $\eta \otimes_{\La} M \in
\Ext_{\La}^{t*} ( _{\psi^*}M, M)$ by a map $f_{\eta} \colon
\Omega_{\La}^{tw}(_{\psi^w}M) \to M$. Then for each $i \geq 0$ there
is a map $f_i \colon _{\psi^w}(P_{tw+i}) \to P_i$ making the diagram
$$\xymatrix{
\cdots \ar[r] & _{\psi^w}(P_{tw+2}) \ar[d]^{f_2} \ar[r]^{d_{tw+2}} &
_{\psi^w}(P_{tw+1}) \ar[d]^{f_1} \ar[r]^{d_{tw+1}} &
_{\psi^w}(P_{tw}) \ar[d]^{f_0} \ar[r] &
\Omega_{\La}^{tw}(_{\psi^w}M) \ar[d]^{f_{\eta}} \ar[r] & 0 \\
\cdots  \ar[r] & P_2 \ar[r]^{d_2} & P_1 \ar[r]^{d_1} & P_0
\ar[r]^{d_0} & M \ar[r] & 0 }$$ with exact rows commute. If $\theta
\in \Ext_{\La}^{t*} ( _{\psi^*}M, \La / \ra )$ is a homogeneous
element, say $|\theta| =tn$, and represented by a map $f_{\theta}
\colon _{\psi^n}(P_{tn}) \to \La / \ra$, then $\theta \eta \in
\Ext_{\La}^{t(w+n)} ( _{\psi^{w+n}}M, \La / \ra )$ is represented by
the composite map
$$_{\psi^{w+n}}(P_{t(w+n)}) \xrightarrow{_{\psi^n}f_{tn}}
{_{\psi^n}(P_{tn})} \xrightarrow{f_{\theta}} \La / \ra.$$

For any $i \geq 0$ the complex $_{\psi^i}\mathbb{P}$ is a minimal
projective resolution of $_{\psi^i}M$, and therefore
$\Ext_{\La}^{ti} ( _{\psi^i}M, \La / \ra ) = \Hom_{\La} (
_{\psi^i}(P_{ti}), \La / \ra )$. Moreover, the multiplication map
$$\Ext_{\La}^{ti} ( _{\psi^i}M, \La / \ra ) \xrightarrow{\cdot \eta}
\Ext_{\La}^{t(i+w)} ( _{\psi^{i+w}}M, \La / \ra )$$ is just the map
\begin{eqnarray*}
\Hom_{\La} ( _{\psi^i}(P_{ti}), \La / \ra ) &
\xrightarrow{(_{\psi^i}f_{ti})^*}
& \Hom_{\La} ( _{\psi^{i+w}}(P_{t(i+w)}), \La / \ra ) \\
g & \mapsto & g \circ _{\psi^i}f_{ti},
\end{eqnarray*}
and since the exact sequence
$$_{\psi^{i+w}}(P_{t(i+w)}) \xrightarrow{_{\psi^i}f_{ti}}
{_{\psi^i}(P_{ti})} \to {_{\psi^i}( \Coker f_{ti} )} \to 0$$ shows
that the kernel of the multiplication map is isomorphic to
$\Hom_{\La} ( _{\psi^i}( \Coker f_{ti} ), \La / \ra )$, we see that
$\Coker f_{ti} =0$ for $i \gg 0$. Consequently, for each $i \gg 0$
there exists a surjective map $\Omega_{\La}^{t(w+i)} ( _{\psi^w}M)
\twoheadrightarrow \Omega_{\La}^{ti}(M)$ and therefore also a
sequence
$$\cdots \twoheadrightarrow \Omega_{\La}^{t(3w+i)} ( _{\psi^{3w}}M)
\twoheadrightarrow \Omega_{\La}^{t(2w+i)} ( _{\psi^{2w}}M)
\twoheadrightarrow \Omega_{\La}^{t(w+i)} ( _{\psi^w}M)
\twoheadrightarrow \Omega_{\La}^{ti}(M)$$ of surjections. However,
by assumption the sequence
$$\ell_k (P_0), \ell_k (P_t), \ell_k (P_{2t}), \dots$$
is bounded, and therefore $\Omega_{\La}^{t((q+1)w+i)} (
_{\psi^{(q+1)w}}M)$ must be isomorphic to $\Omega_{\La}^{t(qw+i)}
( _{\psi^{qw}}M)$ for large $q$ and $i$. By setting $j=qw+i$ and
twisting with the automorphism $\psi^{-qw}$, we see that
$\Omega_{\La}^{tj}(M) \simeq \Omega_{\La}^{t(j+w)}( _{\psi^w}M)$
for $j \gg 0$.
\end{proof}

As a particular case of the proposition we obtain the following
result on $D \Tr$-periodicity over a Frobenius algebra.

\begin{proposition}\label{DTrperiodic}
Suppose $k$ is a field and $\La$ is a Frobenius algebra, and let
$\La \xrightarrow{\nu} \La$ be a Nakayama automorphism. If $M$ does
not have a nonzero projective summand and
\emph{\bf{Fg($M,H,\nu^n,2n$)}} holds for some $n \geq 1$, then $\dim
\V_H^{\nu}(M)=1$ if and only if $M \simeq \tau^p(M)$ for some $p
\geq 1$.
\end{proposition}

\begin{proof}
By the previous proposition the variety $\V_H^{\nu}(M)$ is one
dimensional if and only if there is a positive integer $w$ such that
$\Omega_{\La}^{2nj}(M) \simeq \Omega_{\La}^{2n(j+w)}( _{\nu^{nw}}M)$
for some $j \geq 0$. Taking cosyzygies we see that the latter
happens precisely when $M \simeq \Omega_{\La}^{2nw}( _{\nu^{nw}}M)$,
that is, when $M \simeq \tau^{nw}(M)$.
\end{proof}

We illustrate this last result with an example.

\begin{example}
Let $k$ be an algebraically closed field of odd characteristic and
$q \in k$ a nonzero element which is \emph{not} a root of unity, and
denote by $\La$ the $k$-algebra
$$\La = k \langle x,y \rangle / ( x^2, xy+qyx, y^2 ).$$
We saw in the example following Proposition \ref{commutativity} that
the Nakayama automorphism $\nu$ of $\La$ is defined by
$$x \mapsto -q^{-1}x, \hspace{3mm} y \mapsto -qy,$$
and that $\HH^{4*}(_{\nu^{2*}}\La_1, \La )$ is isomorphic to the
fibre product $k[ \theta ] \times_k k$ with $\theta$ of degree $4$.
In particular $\HH^{4*}(_{\nu^{2*}}\La_1, \La )$ is strongly
commutative, and we denote this ring by $H$.

For elements $\alpha, \beta \in k$, denote by $M_{(\alpha, \beta)}$
the $\La$-module $\La (\alpha x+ \beta y)$ (see \cite{Smalo} for a
counterexample, using the module $M_{(1,1)}$, to a question raised
by Auslander, a question for which a counterexample was first given
in \cite{Jorgensen}). Consider the module $M = M_{(1, \beta )}$ for
$\beta \neq 0$, and for each $i \geq 0$ let $P_i= \La$. The sequence
$$\mathbb{P} \colon \cdots \to P_3 \xrightarrow{\cdot(x+q^3\beta y)} P_2
\xrightarrow{\cdot(x+q^2\beta y)} P_1 \xrightarrow{\cdot(x+q\beta
y)} P_0 \xrightarrow{\cdot(x+\beta y)} M \to 0$$ is a minimal
projective resolution of $M$, hence since $_{\nu^n}\mathbb{P}$ is a
minimal projective resolution of $_{\nu^n}M$ for any $n \geq 1$ we
see that $\Ext_{\La}^{2n}( {_{\nu^n}M},k) = \Hom_{\La} (
_{\nu^n}\La, k)$ is one dimensional. We shall prove that
{\bf{Fg($M,H,\nu,2$)}} holds.

Recall from the example following Proposition \ref{commutativity}
the minimal bimodule projective resolution
$$( \mathbb{F}, \delta ) \colon \cdots \to F^{n+1}
\xrightarrow{\delta} F^n \xrightarrow{\delta} F^{n-1} \to \cdots$$
of $\La$, where the set $\{ \tilde{f}^n_i \}_{i=0}^n$ generates
$F^n$ as a $\Lae$-module and the differential $\delta \colon F^n \to
F^{n-1}$ is given by
$$\tilde{f}^n_i \mapsto \left [ x \tilde{f}^{n-1}_i + (-1)^n q^i
\tilde{f}^{n-1}_i x \right ] + \left [ q^{n-i} y
\tilde{f}^{n-1}_{i-1} + (-1)^n \tilde{f}^{n-1}_{i-1} y \right ].$$
The element $\theta \in H^4$ represented by the map
\begin{eqnarray*}
g_4 \colon _{\nu^2}F^4_1 & \to & \La \\
\tilde{f}^4_i & \mapsto & \left \{ \begin{array}{ll}
                               1 & \text{for } i=2 \\
                               0 & \text{for } i \neq 2
                                  \end{array} \right.
\end{eqnarray*}
generates $H$ as an algebra over $H^0$. The resolution $\mathbb{F}
\otimes_{\La} M$ is also a projective resolution of $M$, and
defining $\La$-linear maps
\begin{eqnarray*}
h_n \colon P_n & \to & F^n \otimes_{\La} M \\
1 & \mapsto & \left ( \sum_{i=0}^n q^{\frac{i(i+1)}{2}} \beta^i
\tilde{f}^n_i \right ) \otimes (x+\beta y)
\end{eqnarray*}
gives a commutative diagram
$$\xymatrix{
\cdots \ar[r] & P_2 \ar[d]^{h_2} \ar[r]^{\cdot (x+q^2\beta y)} & P_1
\ar[d]^{h_1} \ar[r]^{\cdot (x+q\beta y)} &
P_0 \ar[d]^{h_0} \ar[r]^{\cdot (x+\beta y)} & M \ar@{=}[d] \ar[r] & 0 \\
\cdots \ar[r] & F^2 \otimes_{\La} M \ar[r]^{\delta \otimes 1} & F^1
\otimes_{\La} M \ar[r]^{\delta \otimes 1} & F^0 \otimes_{\La} M
\ar[r] & M \ar[r] & 0 }$$ with exact rows. Consequently, the element
$\theta \otimes_{\La} M \in \Ext_{\La}^4 ( _{\nu^2}M,M)$ is
represented by the composite map
\begin{eqnarray*}
\bar{f}_{\theta} \colon _{\nu^2}(P_4) & \xrightarrow{(g_4 \otimes 1)
\circ h_4} & \La \otimes_{\La} M \simeq M  \\
1 & \mapsto & q^3\beta^2(x+\beta y).
\end{eqnarray*}
Now for each $i \geq 0$ define $f_i \colon _{\nu^2}(P_{i+4}) \to
_{\nu^2}(P_i)$ by $1 \mapsto q^{2i+3}\beta^2$. We then obtain a
commutative diagram
$$\xymatrix{
\cdots \ar[r] & _{\nu^2}(P_7) \ar[d]^{f_3} \ar[r]^{\cdot (x+q^7\beta
y)} & _{\nu^2}(P_6) \ar[d]^{f_2} \ar[r]^{\cdot (x+q^6\beta y)} &
_{\nu^2}(P_5) \ar[d]^{f_1} \ar[r]^{\cdot (x+q^5\beta y)} &
_{\nu^2}(P_4)
\ar[d]^{f_0} \ar[r] \ar[dr]^{\bar{f}_{\theta}} & \cdots \\
\cdots \ar[r] & P_3 \ar[r]^{\cdot (x+q^3\beta y)} & P_2
\ar[r]^{\cdot (x+q^2\beta y)} & P_1 \ar[r]^{\cdot (x+q\beta y)} &
P_0 \ar[r]^{\cdot (x+\beta y)} & M \ar[r] & 0 }$$ with exact rows,
hence if $\mu \in \Ext_{\La}^{2n}( _{\nu^n}M,k) = \Hom_{\La} (
_{\nu^n}(P_{2n}), k)$ is an element represented by a map
$\bar{f}_{\mu} \colon _{\nu^n}(P_{2n}) \to k$ we see that the
element $\mu \cdot \theta \in \Ext_{\La}^{2(n+2)}( _{\nu^(n+2)}M,k)$
is represented by the composite map
$$_{\nu^(n+2)}(P_{2(n+2)}) \xrightarrow{_{\nu^n}(f_{2n})}
{_{\nu^n}(P_{2n})} \xrightarrow{\bar{f}_{\mu}} k.$$ Moreover, this
composition is nonzero whenever $\mu$ is nonzero, since $f_{2n}$ is
just multiplication with $q^{4n+3}\beta^2$. Therefore, since
$\Ext_{\La}^{2n}( _{\nu^n}M,k)$ is one dimensional for each $n \geq
1$, the $H$-module $\Ext_{\La}^{2*}( _{\nu^*}M,k)$ is finitely
generated; it is generated as an $H$-module by any $k$-basis in
$\Ext_{\La}^0( _{\nu^0}M,k) = \Hom_{\La}(M,k)$ together with any
nonzero elements $\mu_1 \in \Ext_{\La}^2( _{\nu}M,k)$ and $\mu_2 \in
\Ext_{\La}^4( _{\nu^2}M,k)$.

The above shows that {\bf{Fg($M_{(1, \beta )},H,\nu,2$)}} holds for
$\beta \neq 0$, and since the $2$-complexity of $M_{(1, \beta )}$ is
obviously $1$ we see from Proposition \ref{dimension} that the
variety $\V_H^{\nu}(M_{(1, \beta )})$ is one dimensional. From
Proposition \ref{DTrperiodic} we conclude that $M_{(1, \beta )}
\simeq \tau^w(M_{(1, \beta )})$ for some $w \geq 1$. Indeed, in
\cite{Liu} it is shown that $M_{(1, \beta )}$ is isomorphic to $\tau
(M_{(1, \beta )})$.

Recall that each nonprojective indecomposable $\La$-module is
annihilated by $yx$, and is therefore a module over the algebra $\La
/ (yx)$ (see \cite[Section 4]{Schulz}). The latter is stably
equivalent to the Kronecker algebra, an equivalence under which the
representation
$$\xymatrix{
k \ar@/^/[r]^{\beta} \ar@/_/[r]_{\alpha} & k }$$ corresponds to the
module $M_{(\alpha, \beta )}$. Denoting $M_{(\alpha, \beta )}$ by
$M_{(\alpha, \beta )}^1$, it follows from the well known
representation theory of the Kronecker algebra (see, for example,
\cite{Auslander} and \cite{Assem}) that the indecomposable
$\tau$-periodic $\La$-modules are divided into distinct countable
classes $\{ M_{(\alpha, \beta)}^i \}_{i=1}^{\infty}$, one for each
pair $(\alpha, \beta) \in \{ (0,1) \} \cup \{ (1, \beta) \}_{\beta
\in k}$, such that for each $i \geq 1$ there exists a short exact
sequence
$$0 \to M_{(\alpha, \beta)}^i \to M_{(\alpha, \beta)}^{i-1} \oplus
M_{(\alpha, \beta)}^{i+1} \to M_{(\alpha, \beta)}^i \to 0,$$ where
$M_{(\alpha, \beta)}^0 =0$. Now each such exact sequence induces an
exact sequence
$$\Ext_{\La}^{2*}( _{\nu^*}M_{\alpha}^i,k) \to \Ext_{\La}^{2*}(
_{\nu^*}M_{\alpha}^{i-1},k) \oplus \Ext_{\La}^{2*}(
_{\nu^*}M_{\alpha}^{i+1},k) \to \Ext_{\La}^{2*}(
_{\nu^*}M_{\alpha}^i,k)$$ of $H$-modules, and so since $H$ is
Noetherian and $\Ext_{\La}^{2*}( {_{\nu^*}M}_{(1, \beta )}^1,k)$ is
a finitely generated $H$-module whenever $\beta$ is nonzero, an
induction argument shows that $\Ext_{\La}^{2*}( {_{\nu^*}M}_{(1,
\beta )}^i,k)$ is a finitely generated $H$-module for any $i \geq 1$
and $\beta \neq 0$. We conclude that {\bf{Fg($M_{(1, \beta
)}^i,H,\nu,2$)}} holds for \emph{all} the modules $\{ M_{(1, \beta
)}^i \}_{i=1}^{\infty}$ when $\beta$ is nonzero.

However, there are two more classes of indecomposable
$\tau$-periodic $\La$-modules, namely $\{ M_{(1,0)}^i
\}_{i=0}^{\infty}$ and $\{ M_{(0,1)}^i \}_{i=0}^{\infty}$. Do these
modules satisfy the finite generation hypothesis? The answer is no,
and to see this, consider the module $M_{(1,0)} = \La x$. Letting
$P_i = \La$ for each $i \geq 0$ we see that the sequence
$$\cdots \to P_3 \xrightarrow{\cdot x} P_2 \xrightarrow{\cdot x}
P_1 \xrightarrow{\cdot x} P_0 \xrightarrow{\cdot x} M_{(1,0)} \to
0$$ is a minimal projective resolution of $M_{(1,0)}$, and defining
$\La$-linear maps
\begin{eqnarray*}
h_n \colon P_n & \to & F^n \otimes_{\La} M_{(1,0)} \\
1 & \mapsto & \tilde{f}^n_0 \otimes x
\end{eqnarray*}
gives a commutative diagram
$$\xymatrix{
\cdots \ar[r] & P_1 \ar[d]^{h_1} \ar[r]^{\cdot x} &
P_0 \ar[d]^{h_0} \ar[r]^{\cdot x} & M_{(1,0)} \ar@{=}[d] \ar[r] & 0 \\
\cdots \ar[r] & F^1 \otimes_{\La} M_{(1,0)} \ar[r]^{\delta \otimes
1} & F^0 \otimes_{\La} M_{(1,0)} \ar[r] & M_{(1,0)} \ar[r] & 0 }$$
with exact rows. Since the map $h_4$ does not ``hit" the generator
$\tilde{f}^4_2 \otimes x \in F^4 \otimes_{\La} M_{(1,0)}$, we see
that the element $\theta \otimes_{\La} M \in \Ext_{\La}^4 (
_{\nu^2}M,M)$ is represented by the zero map. This shows that
$\Ext_{\La}^{2*}( {_{\nu^*}M}_{(1,0)}^i,k)$ cannot be a finitely
generated $H$-module, and a similar argument shows that the same is
true for the module $M_{(0,1)}$.

Note also that the finite generation condition cannot hold for any
nonzero indecomposable nonprojective $\La$-module which is not
$\tau$-periodic; if $X$ is such a module and $\Ext_{\La}^{2*}(
_{\nu^*}X,k)$ is finitely generated over $H$, then the rate of
growth of $\Ext_{\La}^{2*}( _{\nu^*}X,k)$ is not more than that of
$H$. However, the latter equals the Krull dimension of $H$, thus
$\gamma \left ( \Ext_{\La}^{2*}( _{\nu^*}X,k) \right ) \leq \gamma
(H) =1$. Since $X$ is nonprojective we conclude that the rate of
growth of $\Ext_{\La}^{2*}( _{\nu^*}X,k)$ is $1$, and so by
Proposition \ref{dimension} the variety $\V_H^{\nu}(X)$ is one
dimensional. But then Proposition \ref{DTrperiodic} implies $X$ is
$\tau$-periodic, a contradiction.
\end{example}

Returning to the general theory, we now impose the finite generation
hypothesis on both $M$ and $\Omega_{\La}^1(M)$. The following result
shows that, in this situation, if the variety of $M$ is nontrivial
(that is, when $M$ does not have finite projective dimension) then
there exists a homogeneous element in $H$ ``cutting down" the
variety by one dimension. Recall first that if $\eta \in
\HH^{t*}(_{\psi^*}\La_1, \La )$ is a homogeneous element, say $\eta
\in \Ext_{\Lae}^{tm}(_{\psi^m}\La_1, \La )$, then it can be
represented by a $\Lae$-linear map $f_{\eta} \colon
\Omega_{\Lae}^{tm} ( _{\psi^m}\La_1 ) \to \La$. This map yields a
commutative diagram
$$\xymatrix{
0 \ar[r] & \Omega_{\Lae}^{tm} ( _{\psi^m}\La_1 ) \ar[r]
\ar[d]^{f_{\eta}} & _{\psi^m}(P^{tm-1}) \ar[d] \ar[r] &
\Omega_{\Lae}^{tm-1} ( _{\psi^m}\La_1 ) \ar@{=}[d] \ar[r] & 0 \\
0 \ar[r] & \La \ar[r] & K_{\eta} \ar[r] & \Omega_{\Lae}^{tm-1} (
_{\psi^m}\La_1 ) \ar[r] & 0 }$$ with exact rows, in which we have
denoted by $P^i$ the $i$th module in the minimal projective bimodule
resolution of $\La$. Note that up to isomorphism the module
$K_{\eta}$ is independent of the map $f_{\eta}$ chosen to represent
$\eta$.

\begin{proposition}\label{reducingdimension}
If both {\bf{Fg($M,H,\psi,t$)}} and
{\bf{Fg($\Omega_{\La}^1(M),H,\psi,t$)}} hold and $M$ does not have
finite projective dimension, then there exists a homogeneous element
$\eta \in H$ of positive degree such that $\dim \V_H^{\psi}
(\Omega_{\Lae}^1(K_{\eta}) \otimes_{\La} M) = \dim \V_H^{\psi} (M)
-1$.
\end{proposition}

\begin{proof}
By assumption the $H$-modules $\Ext_{\La}^{t*} ( _{\psi^*}M, \La /
\ra )$ and $\Ext_{\La}^{t*} ( _{\psi^*}\Omega_{\La}^1(M), \La / \ra
)$ are finitely generated, hence by slightly generalizing the proof
of \cite[Proposition 2.1]{Bergh1} we see that there exists a
homogeneous element $\eta \in H$ of positive degree, say $\eta \in
H^{tm} \subseteq \Ext_{\Lae}^{tm}(_{\psi^m}\La_1, \La )$, such that
the multiplication maps
\begin{eqnarray*}
\Ext_{\La}^{ti} ( _{\psi^i}M, \La / \ra ) & \xrightarrow{\cdot \eta}
& \Ext_{\La}^{t(i+m)} ( _{\psi^{i+m}}M, \La / \ra ) \\
\Ext_{\La}^{ti} ( _{\psi^i}\Omega_{\La}^1(M), \La / \ra ) &
\xrightarrow{\cdot \eta} & \Ext_{\La}^{t(i+m)} (
_{\psi^{i+m}}\Omega_{\La}^1(M), \La / \ra )
\end{eqnarray*}
are both $k$-monomorphisms for $i \gg 0$. Consider the short exact
sequence
$$0 \to \La \to K_{\eta} \to \Omega_{\Lae}^{tm-1} (
_{\psi^m}\La_1 ) \to 0$$ obtained from $\eta$. As
$\Omega_{\Lae}^{tm-1} ( _{\psi^m}\La_1 )$ is right $\La$-projective,
the sequence splits when considered as a sequence of right
$\La$-modules, and consequently the sequence
$$0 \to M \to K_{\eta} \otimes_{\La} M \to \Omega_{\Lae}^{tm-1} (
_{\psi^m}\La_1 ) \otimes_{\La} M \to 0$$ is exact. For each $i \geq
0$ the latter sequence induces a long exact sequence
\begin{eqnarray*}
& \Ext_{\La}^{ti} ( {_{\psi^i}(K_{\eta} \otimes_{\La} M)}, \La / \ra
) \to \Ext_{\La}^{ti} ( {_{\psi^i}M}, \La / \ra )
\xrightarrow{\partial_{ti}} & \\
& \Ext_{\La}^{t(i+m)} ( {_{\psi^{i+m}}M}, \La / \ra ) \to
\Ext_{\La}^{ti+1} ( {_{\psi^i}(K_{\eta} \otimes_{\La} M)}, \La / \ra
) \to & \\
& \Ext_{\La}^{ti+1} ( {_{\psi^i}M}, \La / \ra )
\xrightarrow{\partial_{ti+1}} \Ext_{\La}^{t(i+m)+1} (
{_{\psi^{i+m}}M}, \La / \ra ) &
\end{eqnarray*}
in which we have replaced $\Ext_{\La}^{j} (
{_{\psi^i}(\Omega_{\Lae}^{tm-1} ( _{\psi^m}\La_1 ) \otimes_{\La}
M)}, \La / \ra )$ with $\Ext_{\La}^{j+tm-1} ( {_{\psi^{i+m}}M}, \La
/ \ra )$, due to the fact that ${_{\psi^i}(\Omega_{\Lae}^{tm-1} (
_{\psi^m}\La_1 ) \otimes_{\La} M)}$ is a $(tm-1)$th syzygy of
${_{\psi^{i+m}}M}$. By \cite[Theorem III.9.1]{MacLane} the
connecting homomorphism $\partial_j$ is then the Yoneda product with
the extension $(-1)^j {_{\psi^i}( \eta \otimes_{\La} M)}$, in
particular we see that $\partial_{ti}$ is scalar multiplication with
$(-1)^{ti} \eta$.

Now consider the connecting homomorphism $\partial_{ti+1}$. Applying
the Yoneda relation (\ref{Yonedaext}) from Section \ref{hoc} to
$\eta$ and the short exact sequence
$$\theta \colon 0 \to \Omega_{\La}^1(M) \to P_0 \to M \to 0$$
gives the relation
$$( \eta \otimes_{\La} \Omega_{\La}^1(M) ) \circ ( {_{\psi^m}\La_1}
\otimes_{\La} \theta ) = (-1)^{tm} ( \La \otimes_{\La} \theta )
\circ ( \eta \otimes_{\La} M),$$ which we may twist by $\psi^i$ to
obtain the relation
$${_{\psi^i}( \eta \otimes_{\La} \Omega_{\La}^1(M) )} \circ
{_{\psi^{i+m}}\theta} = (-1)^{tm} {_{\psi^i}\theta} \circ
{_{\psi^i}( \eta \otimes_{\La} M)}.$$ This gives a commutative
diagram
$$\xymatrix{
\Ext_{\La}^{ti}( {_{\psi^i}\Omega_{\La}^1(M)}, \La / \ra )
\ar[d]^{\cdot (-1)^{t(i+m)+1} \eta} \ar[r]^{{_{\psi^i}\theta}} &
\Ext_{\La}^{ti+1}(
{_{\psi^i}M}, \La / \ra ) \ar[d]^{\partial_{ti+1}} \\
\Ext_{\La}^{t(i+m)}( {_{\psi^{i+m}}\Omega_{\La}^1(M)}, \La / \ra )
\ar[r]^{{_{\psi^{i+m}}\theta}} & \Ext_{\La}^{t(i+m)+1}(
{_{\psi^{i+m}}M}, \La / \ra ) }$$ in which the horizontal maps are
isomorphisms, hence the connecting homomorphism $\partial_{ti+1}$ is
also basically just scalar multiplication with $\eta$, as was the
case with $\partial_{ti}$. Consequently they are both injective for
$i \gg 0$, giving a short exact sequence
\begin{equation*}\label{exactseq}
0 \to {_{\La}^{ti}( {_{\psi^i}M}, \La / \ra )} \xrightarrow{\cdot
\eta} {_{\La}^{t(i+m)}( {_{\psi^{i+m}}M}, \La / \ra )} \to
{_{\La}^{ti+1}( {_{\psi^i}(K_{\eta} \otimes_{\La} M)}, \La / \ra )}
\to 0 \tag{$\dagger \dagger$}
\end{equation*}
for large $i$ (in which we have used the short hand notion
${_{\La}^j(-,-)}$ for $\Ext_{\La}^j(-,-)$). Note that we may
identify $\Ext_{\La}^{ti+1} ( {_{\psi^i}(K_{\eta} \otimes_{\La} M),
\La / \ra )}$ with $\Ext_{\La}^{ti} (
{_{\psi^i}(\Omega_{\Lae}^1(K_{\eta}) \otimes_{\La} M)}, \La / \ra
)$; since $K_{\eta}$ is right $\La$-projective the projective
bimodule cover
$$0 \to \Omega_{\Lae}^1(K_{\eta}) \to Q \to K_{\eta} \to 0$$
of $K_{\eta}$ splits as a sequence of right $\La$-modules, and
therefore stays exact when tensored with $M$. In addition, the
$\La$-module $Q \otimes_{\La} M$ is projective, hence
$$\Ext_{\La}^{ti+1} (
{_{\psi^i}(K_{\eta} \otimes_{\La} M), \La / \ra )} \simeq
\Ext_{\La}^{ti} ( {_{\psi^i}(\Omega_{\Lae}^1(K_{\eta}) \otimes_{\La}
M)}, \La / \ra ).$$

Consider now the $H$-module $\Ext_{\La}^{t*} (
{_{\psi^*}(\Omega_{\Lae}^1(K_{\eta}) \otimes_{\La} M)}, \La / \ra
)$, and let $w$ be an integer such that the sequence
(\ref{exactseq}) is exact for $i \geq w$. Then the submodule
$\bigoplus_{i=w}^{\infty} \Ext_{\La}^{ti} (
{_{\psi^i}(\Omega_{\Lae}^1(K_{\eta}) \otimes_{\La} M)}, \La / \ra )$
is finitely generated over $H$, being a factor module of the
submodule $\bigoplus_{i=w}^{\infty} \Ext_{\La}^{t(i+m)} (
{_{\psi^{i+m}}M}, \La / \ra )$ of the finitely generated $H$-module
$\Ext_{\La}^{t*} ( {_{\psi^{*}}M}, \La / \ra )$. But then the
$H$-module $\Ext_{\La}^{t*} ( {_{\psi^*}(\Omega_{\Lae}^1(K_{\eta})
\otimes_{\La} M)}, \La / \ra )$ must be finitely generated itself,
since each graded part $\Ext_{\La}^{tj} (
{_{\psi^j}(\Omega_{\Lae}^1(K_{\eta}) \otimes_{\La} M)}, \La / \ra )$
is finitely generated over $H^0$. Also, from \cite[Proposition
3.1]{Bergh1} we get
$$\gamma \left ( \bigoplus_{i=w}^{\infty}
\Ext_{\La}^{ti} ( {_{\psi^i}(\Omega_{\Lae}^1(K_{\eta}) \otimes_{\La}
M)}, \La / \ra ) \right ) = \gamma \left ( \bigoplus_{i=w}^{\infty}
\Ext_{\La}^{t(i+m)} ( {_{\psi^{i+m}}M}, \La / \ra ) \right ) -1,$$
and since for any $\La$-module $X$ the rate of growth of
$\Ext_{\La}^{t*} ( {_{\psi^{*}}X}, \La / \ra )$ equals that of
$\bigoplus_{i=w}^{\infty} \Ext_{\La}^{ti} ( {_{\psi^{i}}X}, \La /
\ra )$ we get
$$\gamma \left ( \Ext_{\La}^{t*} ( {_{\psi^*}(\Omega_{\Lae}^1(K_{\eta})
\otimes_{\La} M)}, \La / \ra ) \right ) = \gamma \left (
\Ext_{\La}^{t*} ( {_{\psi^{*}}M}, \La / \ra ) \right ) -1.$$
Therefore the equality $\cx^t \left ( \Omega_{\Lae}^1(K_{\eta})
\otimes_{\La} M \right ) = \cx^t M -1$ holds, and so from
Proposition \ref{dimension} we conclude that $\dim \V_H^{\psi}
(\Omega_{\Lae}^1(K_{\eta}) \otimes_{\La} M) = \dim \V_H^{\psi} (M)
-1$.
\end{proof}

Finally we turn to the setting in which {\bf{Fg($X,H,\psi,t$)}}
holds for \emph{all} $\La$-modules $X$, and derive two corollaries
from Proposition \ref{reducingdimension}. Observe first that if
{\bf{Fg($\La / \ra,H,\psi,t$)}} holds, then {\bf{Fg($S,H,\psi,t$)}}
holds for every simple $\La$-module $S$, and so by induction on the
length of a module we see that {\bf{Fg($X,H,\psi,t$)}} holds for
every $\La$-module $X$; namely, if $\ell (X) \geq 2$, choose a
submodule $Y \subset X$ such that $\ell (Y) = \ell (X)-1$. The exact
sequence
$$0 \to Y \to X \to X/Y \to 0$$ induces the exact sequence
$$\Ext_{\La}^{t*} ( _{\psi^*}(X/Y), \La / \ra ) \to \Ext_{\La}^{t*} (
_{\psi^*}X, \La / \ra ) \to \Ext_{\La}^{t*} ( _{\psi^*}Y, \La / \ra
)$$ of $H$-modules, and since the end terms are finite over $H$, so
is the middle term.

\begin{corollary}\label{reduceperiodic}
If {\bf{Fg($\La / \ra,H,\psi,t$)}} holds and $\dim \V_H^{\psi} (M) =
d>0$, then there exist homogeneous elements $\eta_1, \dots,
\eta_{d-1} \in H$ of positive degrees such that the module
$$\Omega_{\Lae}^1(K_{\eta_{d-1}}) \otimes_{\La} \cdots \otimes_{\La}
\Omega_{\Lae}^1(K_{\eta_1}) \otimes_{\La} M$$ is eventually
$\psi^i$-periodic for some $i \geq 1$.
\end{corollary}

\begin{proof}
This is a direct consequence of Proposition \ref{periodic} and
Proposition \ref{reducingdimension}.
\end{proof}

\begin{corollary}\label{reduceDTrperiodic}
Suppose $k$ is a field and $\La$ is a Frobenius algebra, and let
$\La \xrightarrow{\nu} \La$ be a Nakayama automorphism. If
{\bf{Fg($\La / \ra,H,\nu^n,2n$)}} holds for some $n \geq 1$ and
$\dim \V_H^{\nu} (M) = d>0$, then there exist homogeneous elements
$\eta_1, \dots, \eta_{d-1} \in H$ of positive degrees such that
every nonzero nonprojective indecomposable summand of
$$\Omega_{\Lae}^1(K_{\eta_{d-1}}) \otimes_{\La} \cdots \otimes_{\La}
\Omega_{\Lae}^1(K_{\eta_1}) \otimes_{\La} M$$ is $\tau$-periodic.
\end{corollary}

\begin{proof}
This is a direct consequence of Proposition \ref{DTrperiodic} and
Proposition \ref{reducingdimension}.
\end{proof}

\section*{Acknowledgement}

I would like to thank my supervisor {\O}yvind Solberg for valuable
suggestions and comments on this paper.

\end{document}